\pdfoutput=1
\documentclass[11pt,reqno]{amsart}

\usepackage[T1]{fontenc}
\usepackage{amsmath}								
\usepackage{amssymb}
\usepackage{amsthm}
\usepackage{amscd}
\usepackage{amsfonts}
\usepackage{mathtools}
\usepackage{mathptmx}
\usepackage[scaled]{helvet}
\usepackage{microtype,comment}
\usepackage[all]{xy}
\usepackage{euler}
\usepackage{extarrows}
\usepackage[colorlinks, citecolor = blue, linkcolor = blue]{hyperref}
\usepackage[dvipsnames]{xcolor}
\usepackage{blkarray}
\usepackage{tikz}									
\usetikzlibrary{matrix}
\usetikzlibrary{patterns}
\usetikzlibrary{matrix}
\usetikzlibrary{positioning}
\usetikzlibrary{decorations.pathmorphing}
\usetikzlibrary{cd}
\usetikzlibrary{intersections,calc}
\tikzset{dot/.style={circle, fill=black, inner sep=.05cm}}
\usepackage{dsfont}
\usepackage[left=3.5cm,top=3.5cm,right=3.5cm]{geometry}
\usepackage[shortlabels]{enumitem}

\newtheorem{theorem}{Theorem}[section]
\newtheorem{lemma}[theorem]{Lemma}

\newtheorem{corollary}[theorem]{Corollary}

\theoremstyle{definition}
\newtheorem{definition}[theorem]{Definition}

\newtheorem{convention}[theorem]{Convention}
\newtheorem{claim}[theorem]{Claim}

\newtheorem{remark}[theorem]{Remark}

\newcommand{\rcolor}[1]{{\color{blue} #1}}

\newcommand{\R}{\mathds{R}}
\newcommand{\Z}{\mathds{Z}}

\newcommand{\bP}{\mathds{P}}

\newcommand{\calM}{\mathcal{M}}

\newcommand{\Mbar}{\overline{\calM}}

\newcommand\trop{{\mathtt{trop}}}

\DeclareMathOperator{\Pic}{Pic}

\newcommand{\dg}{\mathrm{deg}}

\DeclareMathOperator{\Aut}{\mathrm Aut}

\DeclareMathOperator{\lcm}{lcm}

\DeclareMathOperator{\val}{val}

\newcommand\Mgn[1]{\mathcal M_{#1}}
\newcommand\Mgnbar[1]{\Mbar_{#1}}

\newcommand\Hur[1]{\mathcal H_{#1}}
\newcommand\Adm[1]{ \mathcal{ \overline{H}}_{#1}}

\newcommand\br{\mathrm{br}}
\newcommand\src{\mathrm{src}}

\newcommand{\Ft}{\mathsf{F_t}}
\newcommand{\Fs}{\mathsf{F_s}}

\newcommand{\ft}{\mathsf{f_t}}
\newcommand{\fs}{\mathsf{f_s}}

\newcommand{\gGamma}{\widehat{\Gamma}}
\newcommand{\gT}{\widehat{T}}

\newcommand{\Tev}{{\mathsf{Tev}}}

\newcommand{\mysetminusD}{\hbox{\tikz{\draw[line width=0.6pt,line cap=round] (3pt,0) -- (0,6pt);}}}
\newcommand{\mysetminusT}{\mysetminusD}
\newcommand{\mysetminusS}{\hbox{\tikz{\draw[line width=0.45pt,line cap=round] (2pt,0) -- (0,4pt);}}}
\newcommand{\mysetminusSS}{\hbox{\tikz{\draw[line width=0.4pt,line cap=round] (1.5pt,0) -- (0,3pt);}}}
\newcommand{\mysetminus}{\mathbin{\mathchoice{\mysetminusD}{\mysetminusT}{\mysetminusS}{\mysetminusSS}}}
\renewcommand\setminus\mysetminus
\renewcommand\smallsetminus\mysetminus

\title{Tropical Tevelev Degrees}
\author{Renzo Cavalieri,  Erin Dawson}

\begin{document}

\maketitle
\begin{abstract}
    We define the {\it tropical Tevelev degrees}, $\Tev_g^\trop$,  as the degree of a natural finite morphism between certain tropical moduli spaces, in analogy to the algebraic case.  We  develop an explicit combinatorial construction that computes $\Tev_g^\trop = 2^g$.  We prove that these tropical enumerative invariants agree with their algebraic counterparts,  giving an independent tropical computation of the algebraic degrees $\Tev_g$.
\end{abstract}
\section{Introduction}

\subsection{Statement of results} \label{sec:isr}
This works has two main, synergistic objectives: the first is to introduce a family of tropical enumerative geometric invariants inspired by, and agreeing with a corresponding family of algebraic enumerative invariants. The second is to perform an explicit combinatorial computation of these invariants.

For any non-negative integer $g$, let $d = g+1$ and $n = g+3$. These conditions are chosen so that the dimension of the moduli space  of tropical admissible covers $\Hur{g,d,n}^{\trop}$ (see Definition \ref{admtrop} for details) equals the sum of the dimensions of $\Mgn{g,n}^\trop$ and $\Mgn{0,n}^\trop$. The product of forgetful morphisms $\Fs\times \Ft$, forgetting the cover map but remembering the source and target as $n$-pointed curves is then a finite map, and the {\it tropical Tevelev degree} $\Tev_g^\trop$ is defined to be its degree.
There is an algebraic version of Tevelev degrees which is defined analogously and denoted $\Tev_g$. Our first main result is a correspondence theorem.

\begin{theorem}
\label{thm:correspondence}
For any $g$, we have
\begin{equation}
    \Tev_{g}^\trop = \Tev_{g}.
\end{equation}
\end{theorem}

Computing tropical Tevelev degrees is a combinatorial inverse problem: given a pair $(\overline{\Gamma},\overline{T})$ of general $n$-marked tropical curves of genera $g$ and $0$, one must find all possible tropical covers $\phi:\Gamma\to T$ whose source stabilizes to $\overline{\Gamma}$ and target stabilizes to $T$ when all ends are forgotten except the $n$ marks. While this can be a very complicated task, we choose a 

- pair $(\overline{\Gamma},\overline{T})$\footnote{while the notion of a point in ``special-general'' position seems self-contradictory, it is a commonly used notion in enumerative geometry; in this particular case, what we mean is that the point can be choosen in the relative interior of a maximal dimensional cone in some cone complex which appropriately refines $\Mgn{g,n}^\trop\times\Mgn{0,n}^\trop$.} that allows for a richly combinatorial reconstruction of these inverse images, as well as for the computation of their multiplicities. We obtain our second main result.

\begin{theorem}\label{thm:ttev}
    For any positive integer $g$,
    \[
    \Tev^\trop_g = 2^g.
    \]
\end{theorem}
The following immediate corollary shows that this is an instance in which tropical geometry can provide an independent proof for an algebraic statement.
\begin{corollary}
    Theorem \ref{thm:ttev} and Theorem \ref{thm:correspondence} together  provide a tropical proof of \cite[Theorem 1.13]{Tevelev:2020zux}, showing \(
\Tev_g = 2^g.
    \)
    \end{corollary}
\subsection{Context and connections}
Motivated by physics,  \cite{Tevelev:2020zux} introduced the {\it scattering amplitude map} $\Lambda: \Pic^{g+1}(C)\to M_{0,n}$, where $C$ is a general curve of genus $g$ with $n=g+3$ marked points. The map $\Lambda$ is shown to be finite, and its degree $2^g$ is an invariant depending only on $g$. In \cite{CPS:Tevdeg}, where the name {\it Tevelev degrees} is coined,  the authors interpret the family of  scattering amplitude maps parameterized by genus $g$ curves as the product $\fs\times \ft$ of source and target forgetful morphisms from a Hurwitz space, see Section \ref{sec:aac} for details. One may then naturally compactify the Hurwitz space via admissible covers and the  moduli spaces of curves to spaces of stable curves, and approach the study of the degree of $\fs\times \ft$ via intersection theory. A third persepctive, from \cite{A}, views $\Lambda$ as the evaluation morphism from the space of $g^2_d$'s on the curve $C$ to the GIT quotient $(\bP^1)^{n}//\ \bP GL(2)$, and approach its study via techniques of Schubert calculus. Thinking of $g^2_d$'s as rational functions from $C$, one can again collate all scattering amplitude morphisms to a global map from a space of maps $\calM_{g,n}(\bP^1,d)\to \Mgn{g,n}\times ((\bP^1)^{n}//\ \bP GL(2))$; one can then compactify via stable maps and use Gromov-Witten theoretic techniques to approach the problem \cite{B}.

These different perspectives give rise to three kinds of natural generalizations. First, for any integer $l$, if we set $d = g+1+l$ and $n = g+3+2l$, the maps $\fs\times\ft$ remain finite, and their degrees are studied in \cite{CPS:Tevdeg}. Next, one can require non-generic ramification orders in the space of admissible covers, this case has been studied in \cite{C}. Finally, via the Gromov-Witten theoretic approach, Tevelev degrees to higher dimensional targets have been studied in \cite{E,D}.
We expect many of these generalizations to lend themselves nicely to a tropical approach, and in fact \cite{D} already employed tropical techniques in the study of some two dimensional targets. The second author is currently pursuing a systematic study of the one dimensional target case. We find it valuable to write up the case of the original Tevelev degrees on its own, to showcase in the most transparent way the concreteness, as well as the interesting combinatorics coming from the tropical approach.

\subsection{Philosophy and strategy of the tropical approach} One of the common mantras of tropical techniques in enumerative geometry is that tropical geometry is a powerful tool to organize the combinatorics of the degeneration formula for curves. In light of our better understanding of the relationship beteween algebraic and tropical geometry brought about by their fitting together into the picture of logarithmic geometry, we can now say that the tropical perspective goes well beyond a mere organization of algebraic geometric information. To illustrate this point, let us observe the morphism $\fs\times\ft:\Adm{g,d,n} \to \Mgnbar{g,n} \times \Mgnbar{0,n}$ whose degree gives the Tevelev degree $\Tev_g$. While the inverse image of a generic point in the interior of $\Mgnbar{g,n} \times \Mgnbar{0,n}$ consists of the correct number of points, we have no algebraic geometric technique to even {\it name} general smooth curves, let alone reconstruct covers with specified source and target. The degeneration formula approach is based on the fact that both source and target of $\fs\times \ft$ are stratified spaces, which hands us privileged choices for points: the zero dimensional strata, parameterizing the most degenerate curves. In an ideal world, the inverse image of a zero dimensional stratum would consist of a collection of zero dimensional strata, perhaps lots of them, perhaps to be counted with some multiplicities - and tropical geometry would serve as a book-keeping device for this combinatorics. In reality, in spaces of admissible covers nodes of source curves  and their images do not  smooth independently, which as a consequence causes the inverse images of zero dimensional strata  to be positive dimensional. In addition to solving a combinatorial problem, one has to then also correct for excess intersection. This is the approach of \cite{CPS:Tevdeg}: however, to avoid to deal with the amount of excess intersection coming from zero dimensional strata, they limit themselves to pulling back low codimension strata and then obtain recursions among auxiliary types  of Tevelev degrees; thus they need to enlarge their scope in order eventually solve the recursions for the original invariants. 

Tropical geometry witnesses this failure of transversality as follows: the tropicalized map $\Fs\times \Ft$ is not a strict map of cone complexes, as smaller dimensional cones are mapped to the relative interior of larger dimensional cones. This also suggest a remedy for the situation: appropriate refinements of the cone complex structures of target and source can make $\Fs\times \Ft$ a strict map; this correspond to appropriate birational modifications of the algebraic spaces in such a way that the morphism $\fs\times \ft$ extens to a map with zero dimensional fibers. There are new zero
 dimensional strata, parameterizing curves and covers with logarithmic structures, and the count of their inverse image is now completely combinatorial (with no excess intersection); this is in essence the idea driving the correspondence theorem. A more extensive introduction to this circle of ideas may be found in D. Ranganathan lecture notes collected in \cite{F}. 

A further advantage of the tropical approach is that the identification of the inverse images and their multiplicities may be done entirely in the realm of combinatorics: the tropical information of the logarithmic curves and covers is sufficient for this computation. At this point, the challenge is to choose a maximal dimensional cone of the refinement of the cone complex $\Mgn{g,n}^\trop\times \Mgn{0,n}^\trop$ corresponding to tropical curves for which it is possible to solve the combinatorial inverse problem mentioned in Section \ref{sec:isr}.

The reason that this is more complicated than one initially expects is that for  tropical admissible covers, lengths of edges of source and cover curves are not independent. But since for a top dimensional cone in $\Mgn{g,n}^\trop\times \Mgn{0,n}^\trop$ all lengths should be deformable independently, we must be looking for somewhat exotic tropical admissible covers where no edge of the stabilization of the source curve maps to one or a collection of edges of the stabilization of the target curve.

The strategy we settled on eventually consists in separating genus part of the cover from the marked ends: $\overline{\Gamma}$ contains an umnmarked chain of loops with independent edges, to which is attached a tree with all the marked points. The lengths of the stabilized target $\overline{T}$ are chosen uncomparably longer than the edges of $\overline{\Gamma}$. These choices cause two structural advantages in searching for covers $\phi:\Gamma \to T$ stabilizing to $(\overline{\Gamma}, \overline{T})$:
\begin{enumerate}
    \item the chain of loop structure, together with a Riemann-Hurwitz count shows that the loops of the cover must be formed with a very small number of  transpositions, and two points of very high ramification order: this restricts considerably the number of options of how loops are formed and allows for the classification in Section \ref{sec:genuspart}.
    \item the marked ends being all part of one tree forces their inverse images to lie on different copies of a small number of fragments of the same tree, and this structure allows for the classification in Section \ref{sec:markings}.
\end{enumerate}
In conclusion, the careful choice of the point $p = (\overline{\Gamma}, \overline{T})$ in a maximal cone of the refinement of $\Mgn{g,n}^\trop\times \Mgn{0,n}^\trop$ allows for a concrete combinatorial reconstruction of the inverse images. After having proven the correspondence theorem, this computation of Tevelev degrees is certainly more direct and less sophisticated than the previous proofs in the literature. We hope that this type of approach might extend our reach to similar enumerative geometric problems involving counts of curves related to intersection problems on moduli spaces with high amounts of excess intersection. 

\subsection{Acknowledgements} We are grateful to Alessio Cela, Maria Gillespie, Hannah Markwig, Rahul Pandharipande, Dhruv Ranganathan, and Kris Shaw for interesting discussions related to the project. We also acknowledge support from NSF grant DMS-2100962.

\section{Background}
\subsection{Algebraic Tevelev degrees}\label{sec:aac}
We assume familiarity with the moduli spaces of curves and their Deligne-Mumford compactifications, see \cite{KV:IQC, HM:moc} for  introductory presentations.

By $\Adm{g,d,n}$ we denote the admissible cover compactification (\cite{HM:ac}) of the Hurwitz space whose points parameterize isomorphism classes of covers $\varphi:C\to \bP^1$ such that:
\begin{itemize}
    \item $C$ is a connected smooth curve of genus $g$;
    \item $\varphi$ is a map of degree $d$;
    \item all ramification points of $\varphi$ are simple and marked;
    \item $n$ unramified points of $C$ are marked.
\end{itemize}
The space of admissible covers admits natural source and branch morphisms:
\begin{align}
    \src: \Adm{g,d,n} \to \Mgnbar{g,2g+2d-2+n} \nonumber\\
    \br: \Adm{g,d,n} \to \Mgnbar{0,2g+2d-2+n}
\end{align}

 Define $\fs:= \pi_R\circ \src$ by postcomposing the source morphism with the forgetful morphism forgetting the $2g+2d-2$ marks corresponding to ramification points; similarly, $\ft:= \pi_B\circ \br$  is obtained composing the branch morphism with the morphism forgetting the marks corresponding to branch points.
Consider the map
\begin{equation}
    \label{eq:forgfinitemap}
    \fs\times\ft: \Adm{g,d,n} \to \Mgnbar{g,n} \times \Mgnbar{0,n}.
\end{equation}
When $d = g+1$ and $n = g+3$, $\fs\times\ft$ is a finite morphism of $5g$-dimensional spaces. We observe that in this case the cardinality of the ramification (or equivalently branch) locus of any cover $\varphi:C\to \bP^1$ is $4g$. In \cite[Section 1.2]{CPS:Tevdeg} the {\it Tevelev degree} is defined as:
\begin{equation}
    \label{def:algTev}
    \Tev_g:= \frac{\deg (\fs\times\ft)}{(4g)!}.
\end{equation}

In \cite{Tevelev:2020zux}, an equivalent definition is given for Tevelev degrees (called in that paper degrees of the scattering amplitude map), and in Theorem 1.13 it is proved that $\Tev_g = 2^g$.

\subsection{Tropical admissible covers} \label{sec:tac}
We assume familiarity with moduli spaces of tropical curves (\cite{MikhalkinModuli,m:survey}) and the tropicalization statement from \cite{ACP},  that identifies the cone complex $\Mgn{g,n}^\trop$ with the Berkovich skeleton of the analytification of $\Mgn{g,n}$ as a dense open set inside its Deligne-Mumford compactification $\Mgnbar{g,n}$. Unless otherwise specified, we assume all ends of a tropical curve to be labeled.

Tropical admissible covers were introduced in \cite{Caporaso}; their moduli spaces and a tropicalization statement were studied in \cite{CMRadmissible}. We recall some of the statements that are useful in the remainder of the paper.

A {\it tropical admissible cover} of a rational tropical curve with labeled ends is a  morphism of tropical curves $\phi: \Gamma \to T$ satisfying the following requirements:
\begin{enumerate}
\item $T$ is a stable tree with labeled ends.
    \item Parameterizing the relevant edges by arclength (with $\phi(0_e) = 0_{\phi(e)}$), the restriction of $\phi$ to an edge $e$ is a linear function
\begin{equation}
    \phi|_e:[0, l_e]\to [0, l_{\phi(e)}],
\end{equation}
where $l_x$  denotes the length of the edge $x$. The slope $m_e = l_{\phi(e)}/l_e$ is required to be a positive integer, and it is also called the {\it expansion factor} or the {\it degree} of the edge $e$.
\item The map $\phi$ is {\it harmonic}, i.e. for any vertex $v\in \Gamma$ and pairs of edges $e_1, e_2\in T$ incident to $\phi(v)$, we have:
\begin{equation}
    \label{eq:localdegree}
    \sum_{\tiny{\begin{array}{c}e\ni v\\ \phi(e) = e_1\end{array}}} m_e = 
       \sum_{\tiny{\begin{array}{c}e\ni v\\ \phi(e) = e_2\end{array}}} m_e
\end{equation}
The quantity in \eqref{eq:localdegree} is a well-defined invariant of the vertex $v$, called the {\it local degree} of $\phi$ at $v$. By harmonicity we have a well-defined notion of {\it degree} of $\phi$, which may be defined equivalently as either the sum of the local degrees of all vertices in the inverse image of a given vertex of $T$, or the sum of the expansion factors of all edges in the inverse image of a given edge of $T$.

\item The {\it local Riemann-Hurwitz condition} is satisfied at every vertex $v$ of $\Gamma$, i.e.
\begin{equation}
    \val_v +2g_v-2= d_v(\val_{\phi(v)}-2 );
\end{equation}
here  $\val$ stands for valence, $g_v$ is the genus of the vertex $v$ and $d_v$ is the local degree of $\phi$ at $v$.
\end{enumerate}

The {\it combinatorial type} $\Theta$ of a tropical admissible cover $\phi:\Gamma \to T $ is the data obtained forgetting all metric information for $\Gamma$ and $T$, but remembering  expansion factors of edges of $\Gamma$.

The set of admissible covers of a given combinatorial type $\Theta$ is naturally parameterized by the cone $\sigma_\Theta = \R_{\ge 0}^{|CE(T)|}$, where $CE(T)$ denotes the set of compact edges of $T$. For $e$ any compact edge of $T$, let $M_e:=\lcm(\{m_{e'}| e'\in \Gamma, \phi(e')=e\})$.
We define an integral structure in $\sigma_\Theta$ by requiring the lengths of all compact edges of $\Gamma$ and $T$ to be integers. We obtain:
\begin{equation}
    \Lambda_\Theta:= \R_{\ge 0}^{|CE(T)|} \cap \bigoplus_{e\in CE(T)}M_e\cdot \Z.
\end{equation}

We now define the discrete data that identifies a moduli space of tropical admissible covers. We call {\it Hurwitz data} the tuple $\mathfrak{h} = (g,d,N, \eta_1, \ldots, \eta_N)$, where $g,d, N$ are non-negative integers and the $\eta_i$'s are partitions of the integer $d$. A combinatorial type $\Theta$  of tropical admissible covers satisfies the Hurwitz data $\mathfrak{h}$ if the genus of $\Gamma$ is equal to $g$, $T$ has $N$ labeled ends, the degree of $\phi$ is $d$ and the collection of expansion factors for the ends of $\Gamma$ above the $i$-th marked end of $T$ agrees with the partition $\eta_i$. We refer to $\eta_i$ as the {\it branching data} for the $i$-th end of $T$.

There are finitely many combinatorial types $\Theta$ satisfying a given Hurwitz data $\mathfrak{h}$, we denote this finite set by $\Theta_\mathfrak{h}$. The moduli space of {\it tropical admissible covers of type $\mathfrak{h}$} is the (generalized) cone complex obtained as the colimit 
\begin{equation}
    \Hur{\mathfrak{h}}^\trop = \lim_{\to} \{\sigma_\Theta\}_{\Theta \in \Theta_{\mathfrak{h}}},
\end{equation}
where the face morphisms are given by automorphisms of tropical covers and edge contractions as described in \cite[Section 3.2.5]{CMRadmissible}.

It is convenient to give the moduli space $\Hur{\mathfrak{h}}^\trop$ the structure of a weighted cone complex, by giving maximal dimensional cones $\sigma_\Theta$ the weight
\begin{equation}\label{def:weights}
    w(\Theta):= \frac{1}{|\Aut(\Theta)|}\cdot \prod_{v\in V(\Gamma)} H_v \cdot\prod_{e\in CE(T)}\frac{\prod_{\phi(e') = e}m_{e'}}{M_e},
\end{equation}
where $H_v$ denotes the {\it local Hurwitz number} associated to the vertex $v$ (\cite[Section 3.2.4]{CMRadmissible}).

\subsection{Degrees of tropical morphisms} In this section we recall some standard facts about the notions of degree of maps of tropical objects. For a more comprehensive introduction to tropical intersection theory see, for example, \cite[Section 6.7]{macstu:tropgeom}.

Let $\Sigma_1, \Sigma_2$ be two generalized, weighted cone complexes obtained as colimits of  collections of smooth cones with integral structures (i.e. they are simplicial and the primitive vectors along the rays generate the integral structure). We assume for simplicity of exposition that there are no self-maps (automorphisms of the cone preserving the integral structure) in the systems: else, one  includes a factor of $1/|\Aut(\sigma)|$ in the weight $w(\sigma)$ of each top dimensional cone $\sigma$ and what follows goes through otherwise unchanged. 

Assume $\Sigma_1, \Sigma_2$ are pure dimensional of equal dimension $m$, and let $F:\Sigma_1\to \Sigma _2$ be a strict morphism of generalized cone complexes; in particular, for any maximal cone $\sigma\in \Sigma_1$, its image is a cone $\sigma'\in \Sigma_2$; the restriction
$F|_{\sigma}: \sigma \to \sigma'$ is a linear function that preserves the integral structures in the sense that $F(\Lambda_\sigma)\subseteq \Lambda_{\sigma'}$.

With all this notation in place, we make the following definitions.

\begin{definition}\label{def:degrees}
    For a morphism $F:\Sigma_1\to \Sigma_2$ as introduced in the previous paragraphs, we have the following notions of degree:
    \begin{description}
        \item[local degree at a point in the source] let $x\in \Sigma_1$ be a point in the relative interior of a maximal cone $\sigma\in \Sigma_1$. We define the {local degree at $x$} to be
        \begin{equation} \label{eq:latind}
            \deg_xF:= 
            \left\{
            \begin{array}{cl}
                0 & \mbox{if $\dim F(\sigma) <\dim \sigma$}, \\
              \frac{w(\sigma)}{w(F(\sigma))}[\Lambda_{F(\sigma)}: F(\Lambda_\sigma)]   &  \mbox{if $\dim F(\sigma) =\dim \sigma$}
            \end{array}
            \right.
        \end{equation}
        \item[local degree above a point in the target]  let $y\in \Sigma_2$ be a point in the relative interior of a maximal cone $\sigma'\in \Sigma_2$. We define the {local degree above $y$} to be
        \begin{equation}
            \deg_yF:= \sum_{x\in F^{-1}(y)} \deg_x F.
        \end{equation}
        \item[global degree] if the local degree above a point in the target is independent of the choice of the point, it gives a well-defined notion of global {\it degree}, i.e.
        \begin{equation}
            \deg F:= \deg_y F, \ \  \ \ \mbox{if for any $y'$, $\deg_{y'} F= \deg_y F$}. 
        \end{equation}
           \end{description}
\end{definition}
There are combinatorial criteria that imply that the degree of a morphism of cone complexes is well-defined. The most explored situation is the case in which $\Sigma_1$ and $\Sigma_2$ are fans, i.e. they are embedded in vector spaces from which they inherit their integral structures. Another technique is to show that $F$ arises as the tropicalization of an algebraic map with a well-defined degree.

We conclude this section with an elementary statement about how to compute the lattice index in \eqref{eq:latind}. Let $M_\sigma$ be the $m\times m$ matrix representing the linear function $F|_{\sigma}$ in the bases given by the primitive integral generators of the rays of $\sigma$ and $F(\sigma)$; then
\begin{equation}
  [\Lambda_{F(\sigma)}: F(\Lambda_\sigma)]   = \left|
\det (M_\sigma)
  \right|.
\end{equation}

\section{Tropical Tevelev degrees and correspondence}

\label{sec:troptd}

We make a definition of tropical Tevelev degrees following  the algebraic one from \cite{CPS:Tevdeg}. We will consider the tropical version of the morphism from \eqref{eq:forgfinitemap}: it is a map of weighted cone complexes with integral structures, for which we have defined a notion of local degree above any general point of the target in Definition \ref{def:degrees}. The correspondence theorem (Theorem \ref{thm:correspondence}) implies that this notion gives a well defined global degree, since that's the case on the algebraic side.

For any non-negative integer $g$, consider the Hurwitz data \begin{equation}\label{eq:hd}\mathfrak{h}(g) = (g, d = g+1, N = 5g+3, \eta_1, \ldots, \eta_{5g+3}),\end{equation} where 
\begin{equation}
    \eta_i = 
    \left\{
    \begin{array}{cl}
        (1, \ldots, 1) &  i\leq g+3\\
        (2, 1, \ldots, 1)  &   g+4 \leq i \leq 5g+3
    \end{array}
    \right.
\end{equation}

\begin{definition} \label{admtrop}
    For any non-negative integer $g$ let $n = g+3$ and 
    define by $\Hur{g,d,n}^\trop$ to be the variant of the moduli space of tropical admissible covers $\Hur{\mathfrak{h}(g)}^\trop$ where: 
    \begin{itemize}
    \item $\mathfrak{h}$ is the Hurwitz data from \eqref{eq:hd};
        \item for $i\leq g+3$,  only one end of $\Gamma$ in the inverse image of the $i$-th end of $T$ is marked; 
        \item for $i>g+3$, the  marked ends of $T$ 
 and of their inverse images in $\Gamma$ are not marked. 
    \end{itemize}
    \end{definition}
We are violating here the convention we stated earlier about all ends of tropical curves being labeled; the spaces thus obtained are finite (cone stack) quotients of the tropical admissible cover spaces defined in Section \ref{sec:tac}; in practice this will cost us having to pay  attention to some additional automorphisms, but this choice both simplifies the combinatorics we will encounter and makes the definition of tropical Tevelev degrees more natural. In particular, we make the following convention about local Hurwitz numbers.

\begin{convention}\label{conv:lochurnum}
    Consider an admissible cover $\phi:\Gamma\to T$, and suppose a vertex $v\in \Gamma$ is adjacent to a certain number $n$ of unmarked ends mapping to the same end of $T$. Then the automorphism factor $1/n!$ consisting of permuting these ends is incorporated in the local Hurwitz number, rather than in the automorphism factors. See Figure \ref{fig:lochur} for an illustration. 
 \end{convention}

 \begin{figure}[tb]
     \centering
     \resizebox{.5\textwidth}{!}{\tikzset{every picture/.style={line width=0.75pt}} 

\begin{tikzpicture}[x=0.75pt,y=0.75pt,yscale=-1,xscale=1]

\draw [line width=1.5]    (154.5,80) -- (175.5,128) ;
\draw [line width=1.5]    (175.5,128) -- (173.5,77) ;
\draw [line width=1.5]    (175.5,128) -- (193.5,81) ;
\draw [line width=1.5]    (130.5,130) -- (218.5,129) ;
\draw [color={rgb, 255:red, 208; green, 2; blue, 27 }  ,draw opacity=1 ][line width=1.5]    (175.5,128) -- (209.5,87) ;
\draw [line width=1.5]    (396.5,130) -- (489.5,132) ;
\draw [line width=1.5]    (443.5,88) -- (443,131) ;
\draw [line width=1.5]    (261.5,129) -- (341.5,129.96) ;
\draw [shift={(344.5,130)}, rotate = 180.69] [color={rgb, 255:red, 0; green, 0; blue, 0 }  ][line width=1.5]    (14.21,-4.28) .. controls (9.04,-1.82) and (4.3,-0.39) .. (0,0) .. controls (4.3,0.39) and (9.04,1.82) .. (14.21,4.28)   ;

\draw (140,135.4) node [anchor=north west][inner sep=0.75pt]    {$\textcolor[rgb]{0.04,0.04,0.93}{4}$};
\draw (203,135.4) node [anchor=north west][inner sep=0.75pt]    {$\textcolor[rgb]{0.01,0.01,0.82}{4}$};
\draw (214,64.4) node [anchor=north west][inner sep=0.75pt]    {$\textcolor[rgb]{0.82,0.01,0.11}{1}$};
\draw (294,101.4) node [anchor=north west][inner sep=0.75pt]    {$\phi $};
\draw (170,138.4) node [anchor=north west][inner sep=0.75pt]    {$v$};

\end{tikzpicture}}
     \caption{A local fragment of a tropical cover $\phi:\Gamma \to T$ near a vertex $v$ of local degree $4$. The blue $4$'s denote the expansion factors of the egdes, which are assumed to be compact edges of $\Gamma$. Diagonally, we have ends of $\Gamma$, only one of which is marked (depicted in red). The local Hurwitz number for $v$ is equal to $1$: the triple Hurwitz number $H_0((4),(4),(1,1,1,1))$ with all ends marked is equal to $6$, but we divide by a factor of $6$ corresponding to permuting the black ends. Such factor is incorporated in the local Hurwitz number and is no longer counted as part of the automorphisms of the cover.}
     \label{fig:lochur}
 \end{figure}
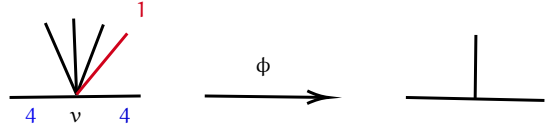

There are natural source and branch morphisms:
\begin{align}
    \src^\trop: \Hur{g,d,n}^\trop \to \Mgn{g,5g+3}^\trop \nonumber\\
    \br^\trop: \Hur{g,d,n}^\trop \to \Mgn{0,5g+3}^\trop
\end{align}

Analogously to the construction in Section \ref{sec:aac}, we postcompose these morphisms with the forgetful morphisms that forget the last $4g$ ends of $T$ and the corresponding marked ends of $\Gamma$ in their inverse image, and obtain forgetful morphisms:
\begin{align}
   \Fs = \pi_{\{i\geq g+4\}}\circ \src^\trop \nonumber\\
   \Ft = \pi_{\{i\geq g+4\}}\circ \br^\trop.
\end{align}

\begin{definition}\label{def:troptev}
For any non-negative
integer $g$, let $d = g+1$ and $n = g+3$.
Consider the morphism of tropical moduli spaces:
\begin{equation}
    \Fs\times \Ft: \Hur{g,d,n}^\trop \to \Mgn{g,n}^\trop \times \Mgn{0,n}^\trop.
\end{equation}
After refining the cone complex structures of source and target, we may assume that $\Fs\times \Ft$ is a strict morphism of generalized cone complexes. We define the {\bf tropical Tevelev degree} to be:
\begin{equation}
    \Tev^\trop_g:={\dg (\Fs\times \Ft)}
\end{equation}
\end{definition}
   
\begin{remark}
Definition \ref{def:troptev} makes sense due to the correspondence theorem (Theorem \ref{thm:correspondence}), which shows that the tropical degree equals the well-defined algebraic degree of the map $\fs\times \ft$.
It  is in fact possible to prove that the  degree of the map $\Fs \times \Ft$ is well-defined directly, combining the perspective and results of \cite{psi-classes, CG:tpsi} with the combinatorial analysis of the map $\trop_\Sigma$ in \eqref{eq:hurwtropdiag}. While this approach would be philosophically more satisfying, it is not strictly necessary for our current purpose and it would take us on a significant technical sidetrack, so we omit it.
\end{remark}

\begin{remark}
 With respect to the algebraic definition, we are missing a denominator of $(4g)!$, since we already chose to unmark the legs corresponding to simple branch points in the tropical Hurwitz space.
\end{remark}

For later convenience, we describe explicitely how to compute the local degree of the map $\Fs\times \Ft$ at a point $x = [\phi:\Gamma\to T]\in \Hur{g,d,n}^\trop$. We refer to \eqref{eq:latind} for the definition of local degree of a map of cone complexes, and  to \eqref{def:weights} for the weights of cones of spaces of tropical admissible covers. For moduli spaces of tropical curves $\Mgn{g,n}^\trop$, maximal cones are weighted by the reciprocal of the size of the automorphism group of the tropical curves parameterized. 
We make the following simplifying assumption, which will be true for the covers considered to compute tropical Tevelev degrees:
\begin{description}
    \item[($\star$)] for each edge $e\in T$, there is at most one edge in $\phi^{-1}(e)$ with expansion factor $>1$.
\end{description}
Assumption $(\star)$ yields the following two consequences:
\begin{enumerate}
    \item The last product in \eqref{def:weights} is  equal to $1$;
    \item  for any compact edge $e$ of $T$, choosing the length of the edge in  $\phi^{-1}(e)$ with highest expansion factor (and choosing any one edge if they all have expansion factor $1$) as a coordinate for $\sigma_\Theta$ gives us a coordinate system whose integral lattice agrees with the integral structure of $\sigma_\Theta$.
 \end{enumerate}
 It follows that the local degree of $\Fs\times \Ft$ at  $x$ is given by:
 \begin{equation}\label{eq:locdegforlater}
     \deg_x (\Fs\times \Ft) =  \frac{|\Aut(\overline{\Gamma})|}{|\Aut(\phi)|}\cdot \prod_{v\in V(\Gamma)} H_v \cdot \left| \det(M_{\sigma_\Theta}) \right|,
 \end{equation}
 where $M_{\sigma_\Theta}$ is the matrix whose rows express the lengths of the compact edges of $\Fs(\Gamma)$ and $\Ft(T)$ as linear functions of the lengths of the edges of $\Gamma$ chosen as discussed in $(2)$ in the previous paragraph.

We are now ready to prove that the tropical enumerative invariants just defined agree with their algebraic counterparts.

\begin{proof}[Proof of Theorem \ref{thm:correspondence}]
    The proof of this theorem is an adaptation of the proof of \cite[Theorem 2]{CMRadmissible}; we describe here the necessary modifications and defer to that paper for some of the details that transfer essentially unchanged.

 Let $\Hur{g,d,n}^{an}$ denote the analytification of the Hurwitz space and by $\Sigma\Hur{g,d,n}^{an}\subset \Hur{g,d,n}^{an}$ the Berkovich skeleton obtained by compactifying the Hurwitz space by admissible covers.
 We have a commutative diagram
 \begin{equation}\label{eq:hurwtropdiag}
     \xymatrix{ 
\Hur{g,d,n}^{an}\ar[r]^{p} \ar@/^2pc/[rr]^{\trop} & \Sigma\Hur{g,d,n}^{an}\ar[r]^{\trop_\Sigma} & \Hur{g,d,n}^\trop},
 \end{equation}
where $p$ is the retraction to the skleleton and the map $\trop_\Sigma$ is a strict morphism of cone complexes which restricts to an isomorphism onto its image for every individual cone, while being globally neither injective nor surjective. Given a maximal dimensional cone $\sigma_\Theta\subset \Hur{g,d,n}^\trop$ corresponding to a combinatrial type $\Theta$ of tropical admissible covers, in \cite[Section 4.2.2, Theorem 2]{CMRadmissible} it is shown that there are exactly $w(\Theta)$ cones in $\Sigma\Hur{g,d,n}^{an}$ mapping isomorphically onto $\sigma_\Theta$. Thus making $\Hur{g,d,n}^\trop$ into a weighted cone complex by giving weight $w(\Theta)$ to each maximal cone $\sigma_\Theta$ makes $\trop_\Sigma$ into a map of weighted cone complexes of degree $1$.
Now consider the diagram:
\begin{equation}
\xymatrix{
\Hur{g,d,n}^{an} \ar[rr]^{(f_t\times f_s)^{an}} \ar[d]_p  & & \Mgn{g,n}^{an} \times \Mgn{0,n}^{an} \ar[d]\\
\Sigma\Hur{g,d,n}^{an} \ar[rr]^{f_t\times f_s\vert_{\Sigma\Hur{}}} \ar[d]_{\trop_\Sigma} & & \Sigma\Mgn{g,n}^{an} \times \Sigma\Mgn{0,n}^{an} \ar@{}[d]|*[@]{=}\\
\Hur{g,d,n}^{an} \ar[rr]^{\Ft\times\Fs}& & \Mgn{g,n}^\trop \times \Mgn{0,n}^\trop.
}
\end{equation}
By the algebraic definition of Tevelev degrees, the map $(f_t\times f_s)^{an}$ is a map of analytic spaces of degree $(4g)! \Tev_g$.
The horizontal map in the middle has a somewhat dual nature: it is naturally the restriction the the skeleton of the previous map, giving rise to an analytic map of the same degree; but it is also a map of cone complexes whose combinatorial degree may be computed combinatorially as in Definition \ref{def:degrees}. After possibly refining the cone complex structures of source and target we can assume the map to be strict. Let $y = (S,T)$ be a general point in $\Sigma\Mgn{g,n}^{an} \times \Sigma\Mgn{0,n}^{an}$, belonging to a maximal cone $\tau_{\overline\Theta}$, and $x = \Gamma \to B$ a point in its inverse image in $\Sigma\Hur{g,d,n}^{an}$; by the assumption of strictness, $x$ is in the interior  of some maximal cone $\sigma_\Theta$. Combining \cite[Section 6]{Rab} with the fact that both points have isotropy given by the automorphism groups of the tropical objects ($(S,T)$ and $\Gamma\to B$), we obtain that the local analytic degree of $f_t\times f_s\vert_{\Sigma\Hur{}}$ at $x$ is given by 
\[
\frac{|\Aut(\overline\Theta)|}{|\Aut(\Theta)|} [\Lambda_{\overline \Theta}: Im(\Lambda_\Theta)],
\]
i.e. it agrees with the combinatorial local degree. Finally, since we have given weights to the cones of $\Hur{g,d,n}^\trop$ to make $\trop_\Sigma$ a map of degree one, we conclude that the combinatorial degree of $\Ft\times\Fs$ equals the analytic degree of $(f_t\times f_s)^{an}$, which immediately implies $\Tev^\trop_g = \Tev_g$.

\end{proof}

\section{Computation of Tropical Tevelev degrees: proof of Theorem \ref{thm:ttev}}

In this section we exhibit a combinatorial computation for tropical Tevelev degrees.
We choose a point $p = (\overline{\Gamma}, \overline{T})$ in the interior of a maximal cone of the refinement of $\Mgn{g,n}^\trop\times \Mgn{0,n}^\trop$ induced by the map $\Fs\times \Ft$. The pair of tropical curves parameterized by $p$ are depicted in Figure \ref{fig:image}. We show that $(\Fs\times\Ft)^{-1}(p)$ consists of $2^g$ points each of multiplicity one. 
We start in Section \ref{sec:examples} by computing the tropical Tevelev degrees for three examples in low genera. In Section \ref{sec:solutions} we show the construction of $2^g$ points of multiplicity one for any $g$. Finally, we end in Section \ref{sec:exclude} by ruling out any other points as preimages of the chosen point $p$. 

\begin{figure}[tb]
    \centering
    \tikzset{every picture/.style={line width=0.75pt}} 

\begin{tikzpicture}[x=0.65pt,y=0.65pt,yscale=-1,xscale=1]

\draw   (23,130) .. controls (23,121.72) and (33.52,115) .. (46.5,115) .. controls (59.48,115) and (70,121.72) .. (70,130) .. controls (70,138.28) and (59.48,145) .. (46.5,145) .. controls (33.52,145) and (23,138.28) .. (23,130) -- cycle ;
\draw    (70,130) -- (105,130) ;
\draw  [draw opacity=0] (105,130) .. controls (105.94,125.43) and (114.27,121.85) .. (124.41,121.85) .. controls (134.48,121.84) and (142.78,125.36) .. (143.81,129.89) -- (124.41,130.85) -- cycle ; \draw   (105,130) .. controls (105.94,125.43) and (114.27,121.85) .. (124.41,121.85) .. controls (134.48,121.84) and (142.78,125.36) .. (143.81,129.89) ;  
\draw  [draw opacity=0] (143.81,129.89) .. controls (143.36,148.42) and (135.25,163.16) .. (125.12,163.28) .. controls (114.54,163.42) and (105.77,147.59) .. (105.53,127.94) .. controls (105.53,127.65) and (105.52,127.37) .. (105.52,127.09) -- (124.67,127.7) -- cycle ; \draw   (143.81,129.89) .. controls (143.36,148.42) and (135.25,163.16) .. (125.12,163.28) .. controls (114.54,163.42) and (105.77,147.59) .. (105.53,127.94) .. controls (105.53,127.65) and (105.52,127.37) .. (105.52,127.09) ;  
\draw    (143.81,129.89) -- (175,130) ;
\draw    (221,130) -- (286,130) ;
\draw [color={rgb, 255:red, 208; green, 2; blue, 27 }  ,draw opacity=1 ]   (253.67,130.67) -- (253.67,109.67) ;
\draw [color={rgb, 255:red, 208; green, 2; blue, 27 }  ,draw opacity=1 ]   (286,130) -- (299,110) ;
\draw [color={rgb, 255:red, 208; green, 2; blue, 27 }  ,draw opacity=1 ]   (286,130) -- (299,149) ;
\draw    (366,131) -- (463,131) ;
\draw [color={rgb, 255:red, 208; green, 2; blue, 27 }  ,draw opacity=1 ]   (366,131) -- (348,112) ;
\draw [color={rgb, 255:red, 208; green, 2; blue, 27 }  ,draw opacity=1 ]   (366,131) -- (347,151) ;
\draw [color={rgb, 255:red, 208; green, 2; blue, 27 }  ,draw opacity=1 ]   (414.5,131) -- (414.5,108) ;
\draw    (524,131) -- (609,131) ;
\draw [color={rgb, 255:red, 208; green, 2; blue, 27 }  ,draw opacity=1 ]   (566.5,131) -- (566,108) ;
\draw [color={rgb, 255:red, 208; green, 2; blue, 27 }  ,draw opacity=1 ]   (609,131) -- (627,111) ;
\draw [color={rgb, 255:red, 208; green, 2; blue, 27 }  ,draw opacity=1 ]   (609,131) -- (626,151) ;

\draw (38.67,98.73) node [anchor=north west][inner sep=0.75pt]  [font=\small]  {$x_{1}$};
\draw (82,111.4) node [anchor=north west][inner sep=0.75pt]  [font=\small]  {$x_{2}$};
\draw (118,104.73) node [anchor=north west][inner sep=0.75pt]  [font=\small]  {$x_{3}$};
\draw (118,165.07) node [anchor=north west][inner sep=0.75pt]  [font=\small]  {$x_{4}$};
\draw (299.86,147.22) node [anchor=north west][inner sep=0.75pt]  [font=\footnotesize,color={rgb, 255:red, 208; green, 2; blue, 27 }  ,opacity=1 ,rotate=-358.45]  {$1$};
\draw (299.86,96.16) node [anchor=north west][inner sep=0.75pt]  [font=\footnotesize,color={rgb, 255:red, 208; green, 2; blue, 27 }  ,opacity=1 ,rotate=-359.63]  {$2$};
\draw (238.68,94.69) node [anchor=north west][inner sep=0.75pt]  [font=\footnotesize,color={rgb, 255:red, 208; green, 2; blue, 27 }  ,opacity=1 ,rotate=-0.17]  {$3$};
\draw (628,154.4) node [anchor=north west][inner sep=0.75pt]  [font=\footnotesize,color={rgb, 255:red, 208; green, 2; blue, 27 }  ,opacity=1 ]  {$1$};
\draw (630.33,93.73) node [anchor=north west][inner sep=0.75pt]  [font=\footnotesize,color={rgb, 255:red, 208; green, 2; blue, 27 }  ,opacity=1 ]  {$2$};
\draw (560,92.73) node [anchor=north west][inner sep=0.75pt]  [font=\footnotesize,color={rgb, 255:red, 208; green, 2; blue, 27 }  ,opacity=1 ]  {$3$};
\draw (334.28,150.79) node [anchor=north west][inner sep=0.75pt]  [font=\footnotesize,color={rgb, 255:red, 208; green, 2; blue, 27 }  ,opacity=1 ,rotate=-359.4]  {$n$};
\draw (332.08,93.85) node [anchor=north west][inner sep=0.75pt]  [font=\footnotesize,color={rgb, 255:red, 208; green, 2; blue, 27 }  ,opacity=1 ,rotate=-0.85]  {$n-1$};
\draw (399.67,94.05) node [anchor=north west][inner sep=0.75pt]  [font=\footnotesize,color={rgb, 255:red, 208; green, 2; blue, 27 }  ,opacity=1 ,rotate=-0.08]  {$n-2$};
\draw (381.33,114.73) node [anchor=north west][inner sep=0.75pt]  [font=\footnotesize]  {$L_{1}$};
\draw (428,115.4) node [anchor=north west][inner sep=0.75pt]  [font=\footnotesize]  {$L_{2}$};
\draw (574,115.4) node [anchor=north west][inner sep=0.75pt]  [font=\footnotesize]  {$L_{g}$};
\draw (181,126.4) node [anchor=north west][inner sep=0.75pt]  [font=\huge]  {$\dotsc $};
\draw (476,127.4) node [anchor=north west][inner sep=0.75pt]  [font=\huge]  {$\dotsc $};
\draw (262,111.73) node [anchor=north west][inner sep=0.75pt]  [font=\small]  {$x_{4g}$};


\draw (160,74.73) node [anchor=north west][inner sep=0.75pt]  [font=\Large]  {$\overline{\Gamma}$};
\draw (460,74.73) node [anchor=north west][inner sep=0.75pt]  [font=\Large]  {$\overline{T}$};

\end{tikzpicture}
    \caption{The graphs $\overline{\Gamma}, \overline{T}$ defining  the chosen point $p$ of $\Mgn{g,n}^\trop\times \Mgn{0,n}^\trop$. The graph $\overline{\Gamma}$, from left to right, consists of a chain of loops, followed by a caterpillar trivalent tree with the marked points in descending order. The $x_i$'s and $L_j$'s label  edge lengths, and the legs in red correspond to the marked points.  We require $x_i<<x_j<<  \ldots <<L_k<<L_l$  for $i<j$ and $k<l$, to ensure that $p$ is in the interior of a maximal cone of the refinement of $\Mgn{g,n}^\trop\times \Mgn{0,n}^\trop$ induced by $\Fs\times \Ft$.}
    \label{fig:image}
\end{figure}
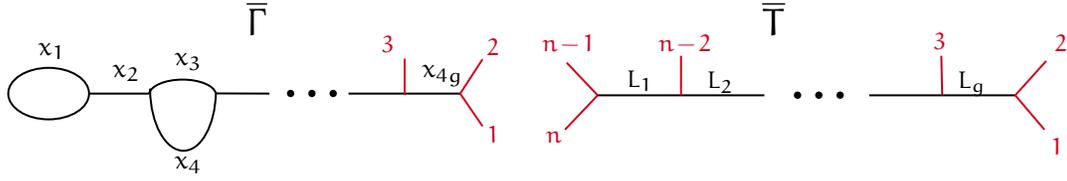

\subsection{Examples in low genera}\label{sec:examples}

\subsubsection{Base case: $g=1$.}\label{sec:g=1}

Recalling the set-up from Section \ref{sec:troptd}, in order to compute $\Tev_1^\trop$ we must compute the degree of the map
\begin{equation}
    \Fs\times \Ft: \Hur{1,2,4}^\trop \to \Mgn{1,4}^\trop \times \Mgn{0,4}^\trop.
\end{equation}

Consider the point $p = (\overline{\Gamma}, \overline{T})\in \Mgn{1,4}^\trop \times \Mgn{0,4}^\trop$ depicted in Figure \ref{fig:pgen1}. The set $(\Fs\times \Ft)^{-1}(p)$ consists of covers $\phi: \Gamma\to T$ such that $T$ stabilizes to 
$\overline{T}$ when forgetting  the four marked ends with branching data $(2)$,  and $\Gamma$ stabilizes to $\overline\Gamma$ when forgetting the four marked ends with expansion factor $2$ as well as all the unmarked ends.

\begin{figure}[b]
    \centering
     \resizebox{\textwidth}{!}{\tikzset{every picture/.style={line width=0.75pt}} 

\begin{tikzpicture}[x=0.75pt,y=0.75pt,yscale=-1,xscale=1]

\draw  [line width=1.5]  (62,135) .. controls (62,129.34) and (66.92,124.75) .. (73,124.75) .. controls (79.08,124.75) and (84,129.34) .. (84,135) .. controls (84,140.66) and (79.08,145.25) .. (73,145.25) .. controls (66.92,145.25) and (62,140.66) .. (62,135) -- cycle ;
\draw [line width=1.5]    (84,135) -- (288,135) ;
\draw [color={rgb, 255:red, 208; green, 2; blue, 27 }  ,draw opacity=1 ][line width=1.5]    (288,135) -- (325,171) ;
\draw [color={rgb, 255:red, 208; green, 2; blue, 27 }  ,draw opacity=1 ][line width=1.5]    (326,107) -- (288,135) ;
\draw [color={rgb, 255:red, 208; green, 2; blue, 27 }  ,draw opacity=1 ][line width=1.5]    (193,134) -- (192,95) ;
\draw [color={rgb, 255:red, 208; green, 2; blue, 27 }  ,draw opacity=1 ][line width=1.5]    (126,135) -- (125,96) ;
\draw [line width=1.5]    (474,139) -- (905,139) ;
\draw [color={rgb, 255:red, 208; green, 2; blue, 27 }  ,draw opacity=1 ][line width=1.5]    (905,139) -- (942,175) ;
\draw [color={rgb, 255:red, 208; green, 2; blue, 27 }  ,draw opacity=1 ][line width=1.5]    (942,106) -- (905,139) ;
\draw [color={rgb, 255:red, 208; green, 2; blue, 27 }  ,draw opacity=1 ][line width=1.5]    (438,103) -- (451.68,116.31) -- (475,139) ;
\draw [color={rgb, 255:red, 208; green, 2; blue, 27 }  ,draw opacity=1 ][line width=1.5]    (475,139) -- (437,167) ;

\draw (333,171.4) node [anchor=north west][inner sep=0.75pt] [font = \Large]   {$\textcolor[rgb]{0.82,0.01,0.11}{1}$};
\draw (953,173.4) node [anchor=north west][inner sep=0.75pt]  [font = \Large]   {$\textcolor[rgb]{0.82,0.01,0.11}{1}$};
\draw (334,94.4) node [anchor=north west][inner sep=0.75pt] [font = \Large]    {$\textcolor[rgb]{0.82,0.01,0.11}{2}$};
\draw (951,89.4) node [anchor=north west][inner sep=0.75pt] [font = \Large]    {$\textcolor[rgb]{0.82,0.01,0.11}{2}$};
\draw (188,68.4) node [anchor=north west][inner sep=0.75pt] [font = \Large]    {$\textcolor[rgb]{0.82,0.01,0.11}{3}$};
\draw (422,171.4) node [anchor=north west][inner sep=0.75pt]  [font = \Large]   {$\textcolor[rgb]{0.82,0.01,0.11}{3}$};
\draw (120,69.4) node [anchor=north west][inner sep=0.75pt]  [font = \Large]   {$\textcolor[rgb]{0.82,0.01,0.11}{4}$};
\draw (419,93.4) node [anchor=north west][inner sep=0.75pt] [font = \Large]    {$\textcolor[rgb]{0.82,0.01,0.11}{4}$};
\draw (64,99.4) node [anchor=north west][inner sep=0.75pt]  [font = \Large]   {$x_1$};
\draw (96,111.4) node [anchor=north west][inner sep=0.75pt] [font = \Large]    {$x_2$};
\draw (154,112.4) node [anchor=north west][inner sep=0.75pt]  [font = \Large]   {$x_3$};
\draw (232,110.4) node [anchor=north west][inner sep=0.75pt]  [font = \Large]   {$x_4$};
\draw (686,116.4) node [anchor=north west][inner sep=0.75pt]  [font = \Large]   {$L_1$};
\draw (166,27.4) node [anchor=north west][inner sep=0.75pt]  [font=\LARGE]  {$\overline{\Gamma}$};
\draw (673,28.4) node [anchor=north west][inner sep=0.75pt]  [font=\LARGE]  {$\overline{T}$};

\end{tikzpicture}}
    \caption{The point $p$ in $\Mgn{1,4}^\trop \times \Mgn{0,4}^\trop$. We have $x_1<<x_2<<x_3<<x_4<<L_1$.}
    \label{fig:pgen1}
\end{figure}
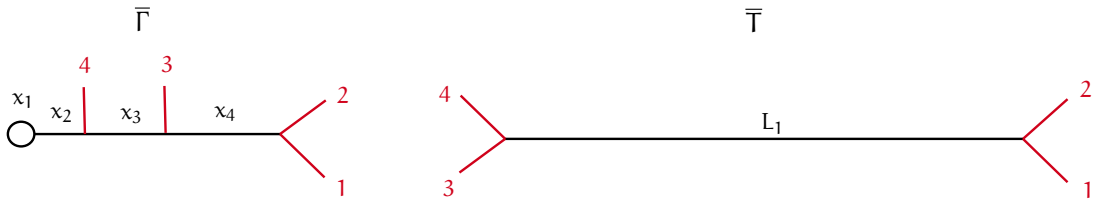

Since $p$ lies in the interior of a maximal cone of (an appropriate refinement of) $\Mgn{1,4}^\trop \times \Mgn{0,4}^\trop$, we know that $T$ is a trivalent tree. Forgetting the four marked ends labeled $1, \ldots, 4$ and their inverse images, we obtain a cover $\tilde{\phi}: \tilde{\Gamma}\to \tilde{T}$, where $\tilde{T}$ is a trivalent tree with four ends assigned branching data $(2)$. There is exactly one Hurwitz cover of $\tilde{T}$,  consisting of two ends with expansion factor one covering the compact edge of $\tilde{T}$ (and forming  a loop between two vertices),  and two infinite edges with degree $2$ coming from each vertex and covering the ends.

We now seek to recover the possible tropical curves $\Gamma$ by adding four marked points.
Since in $\overline\Gamma$ the four marked points belong to a tree that attaches to the loop at a trivalent vertex, we conclude that all $4$ marked points must belong to the same connected component of $\Gamma$ minus the loop, i.e. they must all attach to the same infinite edge of degree two. Due to $L_1>>x_i$ for all $i$, the cover curve $\Gamma$ should contain a long edge mapping to the compact edge of $\overline{T}$ that is lost when stabilizing to $\tilde{\Gamma}$.
The local description of how the two ways this can happen is depicted in Figure \ref{fig:deg2long}.

Given an edge $e$ of the cover of degree $2$, we can attach a trivalent tree to an 
interior point of $\phi(e)$ (on the base), and then mark ends upstairs as on the left hand side of Figure \ref{fig:deg2long}. When stabilized, the cover curve no longer has the length $w$, so $w$ is free to be as long as needed for the compact edge of the base curve to reach length $L_1$.

The only other option is to attach a tripod with two marked points in an interior point of $\phi(e)$,  and then put one marked point on each of the two trees covering the tree in the base, as on the right hand side of Figure \ref{fig:deg2long}. Now the unique length $y$ is lost in the stabilization of $\Gamma$.

From the above local pictures we can get global inverse images of $p$ by placing the marked points on $\tilde{\Gamma}$ to recover $\Gamma$ in two possible ways.
We can attach the tree from the left hand side of Figure \ref{fig:deg2long} at distance $x_2$ from the loop; we can then choose the length $y = x_3, z = x_4 , w = L_1-x_4$.

Alternatively, we place mark $4$ at distance $x_2$ from the loop, followed by mark $3$ after a distance of $x_3$,  and finally  the marks $2$ and $1$ on a tripod (as in Figure \ref{fig:deg2long}) that attaches after distance $x_4$. We  choose $y = L-2x_4$.
Both types of inverse images of $p$ are shown in Figure \ref{fig:genus1}.

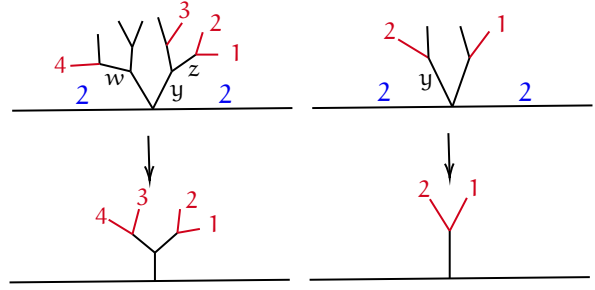
\begin{figure}
    \centering
    \tikzset{every picture/.style={line width=0.75pt}} 

\begin{tikzpicture}[x=0.5pt,y=0.5pt,yscale=-1,xscale=1]

\draw    (334,100.5) -- (546,98.5) ;
\draw    (332,230.5) -- (544,228.5) ;
\draw    (438,229.5) -- (438,193.5) ;
\draw [color={rgb, 255:red, 208; green, 2; blue, 27 }  ,draw opacity=1 ]   (438,193.5) -- (423,169.5) ;
\draw [color={rgb, 255:red, 208; green, 2; blue, 27 }  ,draw opacity=1 ]   (438,193.5) -- (451,168.5) ;
\draw    (437,119.5) -- (437.94,151.5) ;
\draw [shift={(438,153.5)}, rotate = 268.32] [color={rgb, 255:red, 0; green, 0; blue, 0 }  ][line width=0.75]    (10.93,-3.29) .. controls (6.95,-1.4) and (3.31,-0.3) .. (0,0) .. controls (3.31,0.3) and (6.95,1.4) .. (10.93,3.29)   ;
\draw    (440,99.5) -- (422,62.5) ;
\draw    (440,99.5) -- (453,62.5) ;
\draw [color={rgb, 255:red, 208; green, 2; blue, 27 }  ,draw opacity=1 ]   (422,62.5) -- (399,48.5) ;
\draw    (422,62.5) -- (421,38.5) ;
\draw    (453,62.5) -- (446,38.5) ;
\draw [color={rgb, 255:red, 208; green, 2; blue, 27 }  ,draw opacity=1 ]   (453,62.5) -- (468,42.5) ;
\draw [line width=0.75]    (215,231.5) -- (215,210) ;
\draw [line width=0.75]    (215,210) -- (232,195) ;
\draw [line width=0.75]    (215,210) -- (198,196) ;
\draw [color={rgb, 255:red, 208; green, 2; blue, 27 }  ,draw opacity=1 ][line width=0.75]    (198,196) -- (203,177) ;
\draw [color={rgb, 255:red, 208; green, 2; blue, 27 }  ,draw opacity=1 ][line width=0.75]    (198,196) -- (180,185) ;
\draw [color={rgb, 255:red, 208; green, 2; blue, 27 }  ,draw opacity=1 ][line width=0.75]    (232,195) -- (234,177) ;
\draw [color={rgb, 255:red, 208; green, 2; blue, 27 }  ,draw opacity=1 ][line width=0.75]    (232,195) -- (249,192) ;
\draw [line width=0.75]    (213.45,101) -- (196.53,72.75) ;
\draw [line width=0.75]    (213.45,101) -- (227.76,72.75) ;
\draw [line width=0.75]    (196.53,72.75) -- (173.12,67.1) ;
\draw [line width=0.75]    (196.53,72.75) -- (197.83,54.38) ;
\draw [line width=0.75]    (227.76,72.75) -- (223.85,52.97) ;
\draw [line width=0.75]    (227.76,72.75) -- (245.97,60.03) ;
\draw [line width=0.75]    (173.12,67.1) -- (171.82,44.49) ;
\draw [color={rgb, 255:red, 208; green, 2; blue, 27 }  ,draw opacity=1 ][line width=0.75]    (173.12,67.1) -- (151,68.51) ;
\draw [line width=0.75]    (197.83,54.38) -- (187.43,34.6) ;
\draw [line width=0.75]    (197.83,54.38) -- (205.64,34.6) ;
\draw [line width=0.75]    (223.85,52.97) -- (217.35,31.78) ;
\draw [color={rgb, 255:red, 208; green, 2; blue, 27 }  ,draw opacity=1 ][line width=0.75]    (223.85,52.97) -- (235.56,36.02) ;
\draw [color={rgb, 255:red, 208; green, 2; blue, 27 }  ,draw opacity=1 ][line width=0.75]    (245.97,60.03) -- (251.17,43.08) ;
\draw [color={rgb, 255:red, 208; green, 2; blue, 27 }  ,draw opacity=1 ][line width=0.75]    (245.97,60.03) -- (262.88,60.03) ;
\draw    (107,102.5) -- (319,100.5) ;
\draw    (105,232.5) -- (317,230.5) ;
\draw    (210,121.5) -- (210.94,153.5) ;
\draw [shift={(211,155.5)}, rotate = 268.32] [color={rgb, 255:red, 0; green, 0; blue, 0 }  ][line width=0.75]    (10.93,-3.29) .. controls (6.95,-1.4) and (3.31,-0.3) .. (0,0) .. controls (3.31,0.3) and (6.95,1.4) .. (10.93,3.29)   ;

\draw (381,79.4) node [anchor=north west][inner sep=0.75pt]  [font=\small,color={rgb, 255:red, 1; green, 1; blue, 240 }  ,opacity=1 ]  {$2$};
\draw (488,79.4) node [anchor=north west][inner sep=0.75pt]  [font=\small]  {$\textcolor[rgb]{0,0,0.94}{2}$};
\draw (450,152.4) node [anchor=north west][inner sep=0.75pt]  [font=\footnotesize,color={rgb, 255:red, 208; green, 2; blue, 27 }  ,opacity=1 ]  {$1$};
\draw (412,153.4) node [anchor=north west][inner sep=0.75pt]  [font=\footnotesize,color={rgb, 255:red, 208; green, 2; blue, 27 }  ,opacity=1 ]  {$2$};
\draw (471,28.4) node [anchor=north west][inner sep=0.75pt]  [font=\footnotesize,color={rgb, 255:red, 208; green, 2; blue, 27 }  ,opacity=1 ]  {$1$};
\draw (386,33.4) node [anchor=north west][inner sep=0.75pt]  [font=\footnotesize,color={rgb, 255:red, 208; green, 2; blue, 27 }  ,opacity=1 ]  {$2$};
\draw (413,72.4) node [anchor=north west][inner sep=0.75pt]  [font=\footnotesize]  {$y$};
\draw (168,173.4) node [anchor=north west][inner sep=0.75pt]  [font=\footnotesize,color={rgb, 255:red, 208; green, 2; blue, 27 }  ,opacity=1 ]  {$4$};
\draw (199,159.4) node [anchor=north west][inner sep=0.75pt]  [font=\footnotesize,color={rgb, 255:red, 208; green, 2; blue, 27 }  ,opacity=1 ]  {$3$};
\draw (230.86,17.7) node [anchor=north west][inner sep=0.75pt]  [font=\footnotesize,color={rgb, 255:red, 208; green, 2; blue, 27 }  ,opacity=1 ]  {$3$};
\draw (236,161.4) node [anchor=north west][inner sep=0.75pt]  [font=\footnotesize,color={rgb, 255:red, 208; green, 2; blue, 27 }  ,opacity=1 ]  {$2$};
\draw (254.28,27.59) node [anchor=north west][inner sep=0.75pt]  [font=\footnotesize,color={rgb, 255:red, 208; green, 2; blue, 27 }  ,opacity=1 ]  {$2$};
\draw (253,179.4) node [anchor=north west][inner sep=0.75pt]  [font=\footnotesize,color={rgb, 255:red, 208; green, 2; blue, 27 }  ,opacity=1 ]  {$1$};
\draw (269.89,50.19) node [anchor=north west][inner sep=0.75pt]  [font=\footnotesize,color={rgb, 255:red, 208; green, 2; blue, 27 }  ,opacity=1 ]  {$1$};
\draw (224.51,80.69) node [anchor=north west][inner sep=0.75pt]  [font=\footnotesize]  {$y$};
\draw (175.12,72.5) node [anchor=north west][inner sep=0.75pt]  [font=\footnotesize]  {$w$};
\draw (237.86,65.79) node [anchor=north west][inner sep=0.75pt]  [font=\footnotesize]  {$z$};
\draw (153,81.4) node [anchor=north west][inner sep=0.75pt]  [font=\small,color={rgb, 255:red, 1; green, 1; blue, 240 }  ,opacity=1 ]  {$2$};
\draw (261,80.4) node [anchor=north west][inner sep=0.75pt]  [font=\small,color={rgb, 255:red, 1; green, 1; blue, 240 }  ,opacity=1 ]  {$2$};
\draw (137,59.4) node [anchor=north west][inner sep=0.75pt]  [font=\footnotesize,color={rgb, 255:red, 208; green, 2; blue, 27 }  ,opacity=1 ]  {$4$};

\end{tikzpicture}
    \caption{A local picture of two configurations of marked points on a degree two tropical cover that allow for a long edge to be forgotten in the stabilization of the cover curve.}
    \label{fig:deg2long}
\end{figure}

\begin{figure}
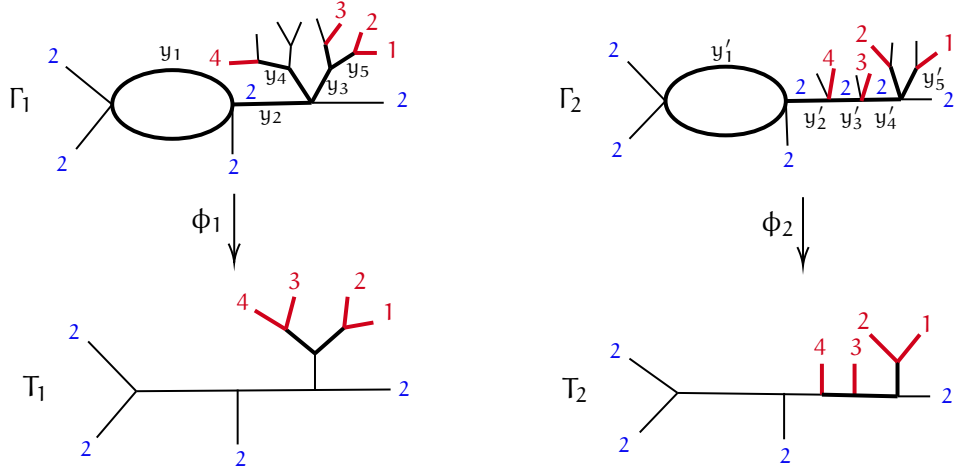

    \centering
    \include{Figures/genus1}
    \caption{The two covers in $(\Fs\times\Ft)^{-1}(p)$. The thickened parts show the subgraphs to which the graphs stabilize via the maps $\Fs, \Ft$. }
    \label{fig:genus1}
\end{figure}

We now compute the local degree of $(\Fs\times\Ft)$ at these inverse images, which gives us the multiplicities with which we need to count the covers.

We follow \eqref{eq:locdegforlater}, and notice that assumption $(\star)$ is verified; hence, for $i = 1,2$ the multiplicity of the inverse image $\phi_i:\Gamma_i\to T_i$ is the product of three factors: an automorphism factor, a product of local Hurwitz numbers and a dilation factor corresponding to the determinant of  the matrix representing the map $\Fs\times\Ft$.

The automorphism factor is the same for $i = 1,2$.
For each cover, we have $\Aut(\phi_i) = \mu_2\times \mu_2$: one factor corresponds to switching simultaneously the two unlabeled left ends of branching type $2$ and their inverse images; the second factor consists of switching the two degree $1$ edges of $\Gamma_i$ forming the loop. However, this is also a nontrivial automorphism of $\overline{\Gamma_i}$. Altogether, we have
\begin{equation}
    \frac{|\Aut(\overline{\Gamma_i})|}{|\Aut(\phi_i)|} = \frac{1}{2}.
\end{equation}

The local Hurwitz numbers factors are also computed identically for $i = 1,2$:  every vertex in the cover is either trivalent with degree one edges in all directions or two edges of degree 2 in different directions and 2 edges of degree 1 in the same direction. Both of these types of vertices have local Hurwitz number equal to 1, therefore the product of all local Hurwitz numbers is 1. 

 To calculate the dilation factors, we set up the following matrices representing the $x_i$'s and $L_j$'s in terms of the $y_k$'s:
\[\begin{array}{cc}
  M_1=\begin{blockarray}{cccccc}
   &y_1 & y_2 & y_3 & y_4 & y_5\\
    \begin{block}{c[ccccc]}
      x_1 & 2 & 0 & 0 & 0 & 0\\
      x_2 & 0 & 1 & 0 & 0 & 0\\
      x_3 & 0 & 0 & 1 & 0 & 0\\
      x_4 & 0 & 0 & 0 & 0 & 1\\
      L_1 & 0 & 0 & 0 & 1 & 1\\
    \end{block}
  \end{blockarray} &
 M_2=\begin{blockarray}{cccccc}
   &y'_1 & y'_2 & y'_3 & y'_4 & y'_5\\
    \begin{block}{c[ccccc]}
      x_1 & 2 & 0 & 0 & 0 & 0\\
      x_2 & 0 & 1 & 0 & 0 & 0\\
      x_3 & 0 & 0 & 1 & 0 & 0\\
      x_4 & 0 & 0 & 0 & 1 & 0\\
      L_1 & 0 & 0 & 0 & 1 & 1\\
    \end{block}
  \end{blockarray}
  \end{array}
\]
We see that $|\det M_i |=2$ for $i = 1,2$.
All together the multiplicity of each cover in $(\Fs\times\Ft)^{-1}(p)$ is $1\cdot \frac{1}{2}\cdot 2=1$. Since we have two inverse images each with multiplicity one,  we obtain $\Tev^\trop_1 = 2$.

\subsubsection{Low genus examples: $g=2$.}
To compute $\Tev_2^\trop$, we look for the degree of the map
\begin{equation}
    \Fs\times \Ft: \Hur{2,3,5}^\trop \to \Mgn{2,5}^\trop \times \Mgn{0,5}^\trop.
\end{equation}

We consider the point $p=(\overline{\Gamma},\overline{T})\in \Mgn{2,5}^\trop \times \Mgn{0,5}^\trop$ as illustrated in Figure \ref{fig:image}. With notation as in the previous section, we know that $\tilde{T}$ is a trivalent tree, but unlike the genus 1 case, there are multiple options for how this tree can be shaped. For two of these $\tilde{T}$-trees it is  possible to place fragments of the marked tree $\overline{T}$ to obtain the base $T$ of the cover. 

When adding the 5 marked points, we use the same techniques as in the previous section to get two long edges of $\Gamma$ that map to  compact edges of $\overline{T}$ and are lost when stabilizing to $\overline{\Gamma}$. All of the preimages are shown in Figure \ref{fig:genus2}.

\begin{figure}[h]
    \centering
    \include{Figures/genus2}
    \caption{The four covers in $(\Fs\times\Ft)^{-1}(p)$ when $g=2$. The thickened parts show the subgraphs to which the graphs stabilize via the maps $\Fs, \Ft$.}
    \label{fig:genus2}
\end{figure}

We can now find the multiplicities of these covers by computing the local degree of $(\Fs\times\Ft)$ at all 4 inverse images. The automorphism factor is the same for all 4. For each cover, there are two pairs of ends of branching type 2 that can be switched, corresponding to the first two forks starting from the left side of the base cover; these give 4 automorphisms of $\phi$ that are not also automorphisms of $\overline{\Gamma}$. The product of local Hurwitz numbers for all covers is equal to 1, recall Convention \ref{conv:lochurnum}. 

\begin{figure}
    \centering
     \resizebox{\textwidth}{!}{\input{Figures/matricesg2}}
    \caption{The matrices computing the local degree of $(\Fs\times\Ft)$ at the four inverse images of $p$. We observe that the matrices are block diagonal, with one block having determinant a power of $2$ and the other determinant one.}
    \label{fig:matricesg2}
\end{figure}

The matrices that compute the multiplicities for these inverse images are written in Figure \ref{fig:matricesg2}.
All four matrices are block diagonal and have absolute value of the determinant equal to $4$, therefore all four covers have multiplicity $\frac{4}{4} =  1$, and we get $\Tev_2^\trop=4$. 

\subsubsection{Low genus examples: $g=3$.}
For our final example, consider the point $p=(\overline{\Gamma},\overline{T})\in \Mgn{3,6}^\trop\times\Mgn{0,6}^\trop$ shown in Figure \ref{fig:image}. There are eight preimages as shown in Figure \ref{fig:genus3}, each has multiplicity $1$, so $\Tev_3^\trop=8.$
\begin{figure}
    \centering
     \resizebox{.8 \textwidth}{!}{\input{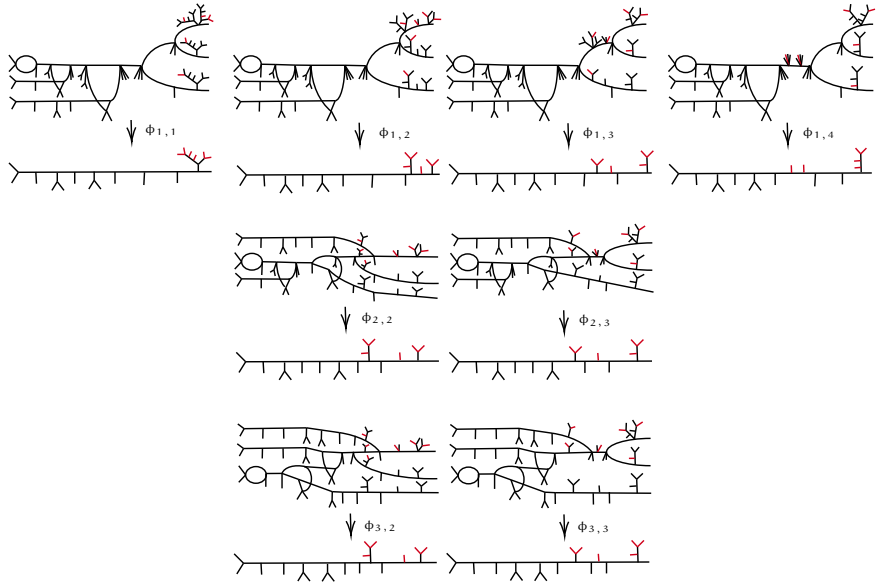}}
    \caption{Genus 3 covers contributing to $\Tev_3^\trop$}
    \label{fig:genus3}
\end{figure}
\subsection{Construction of $2^g$ solutions}

\label{sec:solutions} 
In this section we make explicit and generalize the constructions from the examples in Section \ref{sec:examples}, and construct $2^g$ preimages of $p$ for any genus $g$. We organize the task into four parts: in Section \ref{sec:genuspart} we build a cover containing all the genus and no marked points (we call it the genus part of  the cover and denote it by $\gGamma\to \gT$); in Section \ref{sec:markings}  we construct a tree containing the $n$ marked  ends that attaches to the genus part of $\Gamma$, which we call the marked tree part of the cover. In Section \ref{sec:mult} we show that the multiplicity of every cover constructed is equal to $1$. We conclude in Section \ref{sec:combinatorics} by organizing the combinatorics of the problem and show we have constructed $2^g$ covers contributing to $\Tev_g$.

\subsubsection{The genus part}
\label{sec:genuspart}

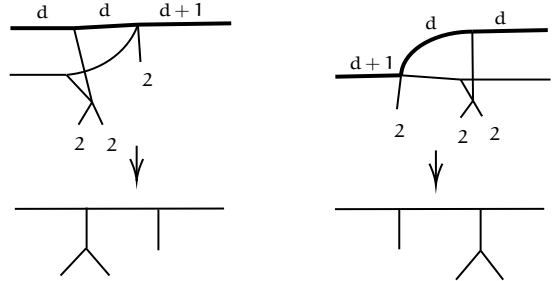
\begin{figure}
    \centering
    \tikzset{every picture/.style={line width=0.75pt}} 

\begin{tikzpicture}[x=0.75pt,y=0.75pt,yscale=-1,xscale=1]

\draw    (26.81,130.05) -- (132.34,129.89) ;
\draw    (188.28,130.2) -- (293.82,130.04) ;
\draw    (63.02,129.4) -- (63.02,149.97) ;
\draw    (62.75,149.97) -- (49.88,163.69) ;
\draw    (62.75,149.97) -- (74.55,163.69) ;
\draw    (99.23,130.38) -- (99.23,150.95) ;
\draw    (261.74,130.53) -- (261.74,151.1) ;
\draw    (261.74,151.1) -- (249.33,165.8) ;
\draw    (261.74,151.1) -- (273.12,165.8) ;
\draw    (220.7,129.88) -- (220.7,150.45) ;
\draw [line width=1.5]    (24.6,38.85) -- (56.63,38.85) ;
\draw [line width=0.75]    (24.22,62.4) -- (52.82,62.4) ;
\draw [line width=0.75]    (56.63,38.85) -- (65.78,76.21) ;
\draw [line width=0.75]    (52.82,62.4) -- (65.78,76.21) ;
\draw [line width=1.5]    (88.89,37.07) -- (135.95,36.41) ;
\draw  [draw opacity=0][line width=0.75]  (89.08,36.53) .. controls (84.44,50.48) and (70.34,60.81) .. (53.2,62.38) -- (48,26.02) -- cycle ; \draw  [line width=0.75]  (89.08,36.53) .. controls (84.44,50.48) and (70.34,60.81) .. (53.2,62.38) ;  
\draw [line width=1.5]    (56.63,38.85) -- (88.89,37.07) ;
\draw [line width=0.75]    (65.78,76.21) -- (58.92,87.58) ;
\draw [line width=0.75]    (65.78,76.21) -- (71.12,87.58) ;
\draw [line width=0.75]    (88.89,37.07) -- (90.19,55.9) ;
\draw [line width=0.75]    (88.26,98.18) -- (88.26,113.82) ;
\draw [shift={(88.26,115.82)}, rotate = 270] [color={rgb, 255:red, 0; green, 0; blue, 0 }  ][line width=0.75]    (10.93,-3.29) .. controls (6.95,-1.4) and (3.31,-0.3) .. (0,0) .. controls (3.31,0.3) and (6.95,1.4) .. (10.93,3.29)   ;
\draw [line width=1.5]    (188.46,63.35) -- (221.61,62.64) ;
\draw [line width=0.75]    (221.61,62.64) -- (219.52,79.76) ;
\draw  [draw opacity=0][line width=1.5]  (221.61,62.65) .. controls (222.2,50.4) and (238.07,40.58) .. (257.55,40.58) .. controls (257.55,40.58) and (257.55,40.58) .. (257.55,40.58) -- (257.55,63.36) -- cycle ; \draw  [line width=1.5]  (221.61,62.65) .. controls (222.2,50.4) and (238.07,40.58) .. (257.55,40.58) .. controls (257.55,40.58) and (257.55,40.58) .. (257.55,40.58) ;  
\draw [line width=0.75]    (221.61,62.64) -- (251.63,64.78) ;
\draw [line width=0.75]    (257.55,40.58) -- (257.91,75.48) ;
\draw [line width=0.75]    (251.63,64.78) -- (257.91,75.48) ;
\draw [line width=0.75]    (251.63,64.78) -- (297,64.78) ;
\draw [line width=1.5]    (257.55,40.58) -- (295.6,39.81) ;
\draw [line width=0.75]    (257.91,75.48) -- (251.63,84.04) ;
\draw [line width=0.75]    (257.91,75.48) -- (262.49,83.32) ;
\draw [line width=0.75]    (239.18,100.55) -- (239.18,116.18) ;
\draw [shift={(239.18,118.18)}, rotate = 270] [color={rgb, 255:red, 0; green, 0; blue, 0 }  ][line width=0.75]    (10.93,-3.29) .. controls (6.95,-1.4) and (3.31,-0.3) .. (0,0) .. controls (3.31,0.3) and (6.95,1.4) .. (10.93,3.29)   ;

\draw (35.52,25.92) node [anchor=north west][inner sep=0.75pt]  [font=\tiny,color={rgb, 255:red, 0; green, 0; blue, 0 }  ,opacity=1 ]  {$d$};
\draw (97.48,24.85) node [anchor=north west][inner sep=0.75pt]  [font=\tiny,color={rgb, 255:red, 0; green, 0; blue, 0 }  ,opacity=1 ]  {$d+1$};
\draw (66.93,24.85) node [anchor=north west][inner sep=0.75pt]  [font=\tiny,color={rgb, 255:red, 0; green, 0; blue, 0 }  ,opacity=1 ]  {$d$};
\draw (193.67,51.27) node [anchor=north west][inner sep=0.75pt]  [font=\tiny,color={rgb, 255:red, 0; green, 0; blue, 0 }  ,opacity=1 ]  {$d+1$};
\draw (231.66,31.13) node [anchor=north west][inner sep=0.75pt]  [font=\tiny,color={rgb, 255:red, 0; green, 0; blue, 0 }  ,opacity=1 ]  {$d$};
\draw (266.49,28.8) node [anchor=north west][inner sep=0.75pt]  [font=\tiny,color={rgb, 255:red, 0; green, 0; blue, 0 }  ,opacity=1 ]  {$d$};
\draw (55.09,92.19) node [anchor=north west][inner sep=0.75pt]  [font=\tiny]  {$2$};
\draw (71.35,92.76) node [anchor=north west][inner sep=0.75pt]  [font=\tiny]  {$2$};
\draw (89.29,60.3) node [anchor=north west][inner sep=0.75pt]  [font=\tiny]  {$2$};
\draw (216.51,85.79) node [anchor=north west][inner sep=0.75pt]  [font=\tiny]  {$2$};
\draw (248.47,88.64) node [anchor=north west][inner sep=0.75pt]  [font=\tiny]  {$2$};
\draw (264.58,87.23) node [anchor=north west][inner sep=0.75pt]  [font=\tiny]  {$2$};

\end{tikzpicture}
    \caption{Two possible ways to add genus. We omit from the picture ends of degree $1$ to avoid clutter. The active path is thickened. We remark that the two fragments are just the reflection	of one another about a vertical axis, but the two distinct directions play an important role in our story.}
    \label{fig:2genusoptions}
\end{figure}
For any  integer $d\geq 1$, we consider the two tropical covers of degree $d+1$ depicted in 
Figure \ref{fig:2genusoptions}, which we call $U$ and $D$, and together we call {\it genus fragments}. The base of each cover is a trivalent tree with five ends. We think of this tree as an oriented line to which we attach one simple end and a tripod, in the two possible orders. 
The cover curves are connected genus one tropical curves; for every horizontal end/edge of the base curve there is a unique end/edge in the cover mapping to it with degree greater than one; we call it the {\it active end/edge} and we call the collection of all active ends/edges the {\it active path}. In the case of fragment $D$ when $d=1$, we choose an arbitrary connected lift of the horizontal path of the base curve and declare it to be the active path.

Notice that in the two cases the degree of the active path either goes up or down by one after making the genus. 

For any genus $g\geq 2$, we construct (topological types of) tropical covers of genus $g$, degree $d\leq g+1$ as follows:
\begin{enumerate}
\item Start with the unique degree $2$,  genus one cover of a trivalent, four ended tree as in Section \ref{sec:g=1}; choose any one of the four ends to be the active edge;
\item attach a sequence of $g-1$ genus fragments by gluing the rightmost horizontal end of the base graph to the leftmost end of the next graph,  and gluing the corresponding active edges above. Observe that the fragment $U$ can always be used,  whereas the fragment $D$ can be used whenever the degree of the attaching active end is greater than one.
\item complete the resulting graphs with appropriate portions of degree $1$ to make it into an honest global cover of tropical curves. This can be always done in a unique way.
\end{enumerate}
If the fragment $D$ has been used $i$ times we obtain a connected cover of degree $g+1-i$ with  active end of degree $g+1-2i$. We complete it to a degree $g+1$ cover by adding $i$ disjoint copies of the  base curve each mapping with degree one.

We conclude this section with the elementary observation that for all covers thus constructed the degree of the active end has the same parity as $g+1$. 

We have constructed tropical	covers containing all the genus and none of the marked points needed: we call these covers the {\it genus part} of our solutions and denote them by $\gGamma \to \gT$. 

\subsubsection{The marked tree part}
\label{sec:markings}

Given a cover $\Gamma\to T$, we denote by $\tilde{\Gamma}\to \tilde{T}$ the cover obtained by forgetting the $n$ marked points. For any genus part of a cover $\gGamma\to \gT$ constructed in the previous section, 
the degree of the active end can be any positive number congruent to $d$ modulo $2$. Given a genus part with active edge of degree $d-2i$, we show first how to complete it to unmarked covers $\tilde{\Gamma}\to \tilde{T}$; next we add further trees containing the (images of the) marked points to $\tilde{T}$ to obtain $T$; finally we describe the inverse images of these trees and describe how to place the marked points on them.

\subsubsection*{Cuts and joins: constructing $\tilde{\Gamma}\to \tilde{T}$}

For any genus part of a cover, we complete $\gT$ to $\tilde{T}$ by extending horizontally the end that the active edge maps to and attaching to it $g-1$ vertical ends. Together with the final horizontal end, we have added $g$ simple branched ends to $\gT$ and therefore obtained a good candidate for $\tilde{T}$.
A consequence of the simple branching conditions is that for any cover of $\tilde\Gamma\to \tilde{T}$, the inverse image of the horizontal edge of $\tilde{T}$ must be a trivalent tree, hence it should be obtained from the active edge by a series of {\it cuts and joins} (illustrated in Figure \ref{fig:cutjoin}).

\begin{figure}[tb]
    \centering
    \tikzset{every picture/.style={line width=0.75pt}} 

\begin{tikzpicture}[x=0.5pt,y=0.5pt,yscale=-1,xscale=1]

\draw [line width=1.5]    (90.5,138) -- (300.5,138) ;
\draw [line width=1.5]    (368,138) -- (579,138.5) ;
\draw  [draw opacity=0] (91.56,71.61) .. controls (93.61,71.54) and (95.67,71.5) .. (97.75,71.5) .. controls (156.17,71.5) and (203.57,101.23) .. (204,138) -- (97.75,138.5) -- cycle ; \draw   (91.56,71.61) .. controls (93.61,71.54) and (95.67,71.5) .. (97.75,71.5) .. controls (156.17,71.5) and (203.57,101.23) .. (204,138) ;  
\draw  [draw opacity=0] (577.74,204.4) .. controls (575.69,204.48) and (573.62,204.53) .. (571.54,204.54) .. controls (513.13,204.79) and (465.59,175.27) .. (465.01,138.5) -- (571.25,137.54) -- cycle ; \draw   (577.74,204.4) .. controls (575.69,204.48) and (573.62,204.53) .. (571.54,204.54) .. controls (513.13,204.79) and (465.59,175.27) .. (465.01,138.5) ;

\end{tikzpicture}
    \caption{The active edge, drawn thickened, is oriented from left to right, i.e. away from the genus part. An edge joining the active edge is shown on the left, and an edge cutting from the active edge is shown on the right.}
    \label{fig:cutjoin}
\end{figure}
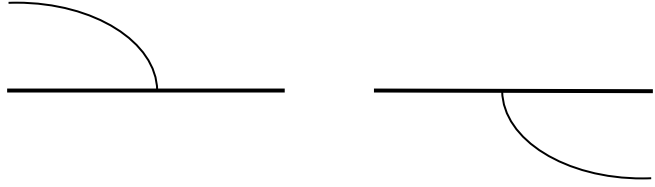

Consider a genus part $\gGamma\to \gT$ where the active edge has degree $d$: then we obtain a unique cover $\tilde{\Gamma}\to \tilde{T}$ by a sequence of cuts where one of the ends has expansion factor equal to $1$, and the other (of degree greater than one) is the new {\it active edge}. We call the sequence of active edges the {\it active path}. Observe that since we have $g-1$ cuts, the final active end has degree $d-(g-1) = 2$, and is therefore a simply ramified end, as required. We observe that since all the non-active ends have degree $1$, there is a unique way to extend them to a cover of the portion of the tree $\tilde{T}$ that they must cover.

For $1\leq i\leq \lfloor \frac{d-1}{2} \rfloor$, we complete the genus part $\gGamma_i\to \gT_i$ in $d-2i$ different ways: informally, the idea is that we keep all the joins together. 
Formally, for $0\leq k\leq d-2i-1$, we denote by $\tilde{\Gamma}_{i,k}
\to \tilde{T}_{i,k}$ the cover where the active path consists of $k$ cuts, followed by $i$ joins, followed by the remaining cuts, as illustrated in Figure \ref{fig:covers}. We remark again that for each of the degree one edges emanating from the active path, there is a unique way to complete them to a degree one cover of the portion of $\tilde{T}$ they must cover.


\begin{figure}
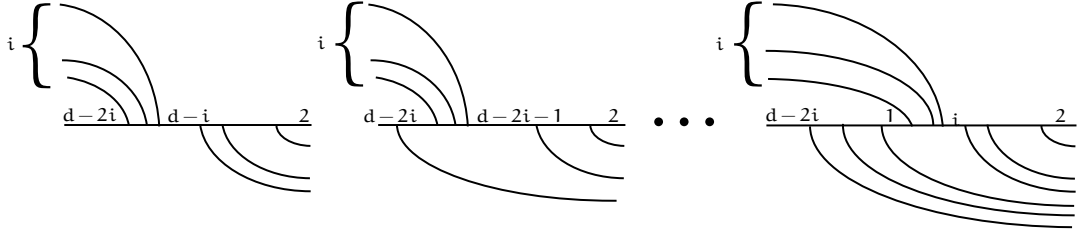

    \centering
    \include{Figures/covers}
    \caption{The sequence of cuts and joins admitted for a given $i$.}
    \label{fig:covers}
\end{figure}

\subsubsection*{The marking fragments}

In the next section,  we will construct $T$ from $\tilde{T}$ by attaching to the horizontal edge of $\tilde{T}$ one of the $(n-2)$ marked fragments depicted in Figure \ref{fig:markfrag}. Note that since we are in the base cover, the marks are really representing the images of the marks.
Finally observe that $F_1$ contains the tree $\overline{T}$; for $1\leq j\leq n-3$, the tree $\overline{T}$ is formed by the fragment together with the portion of the horizontal edge of $\tilde{T}$ between the two outer connected components of the fragment. For $j = n-2$, when stabilizing, the $n$-th marked end will be adjacent to the $(n-1)$-th at a vertex $v$, and $\overline{T}$ will consist of these two ends, the portion of the horizontal edge from $v$ to the remaining part of the fragment, and the remaining part of the fragment.
For now, we don't discuss how the connected components of the fragments are positioned with respect to the simple branched ends stemming from the horizontal edge of $\overline{T}$. We will discuss the various possible cases when we describe the inverse images of the fragments to construct the cover $\Gamma$.

%



\begin{figure}
    \centering
     \resizebox{\textwidth}{!}{\input{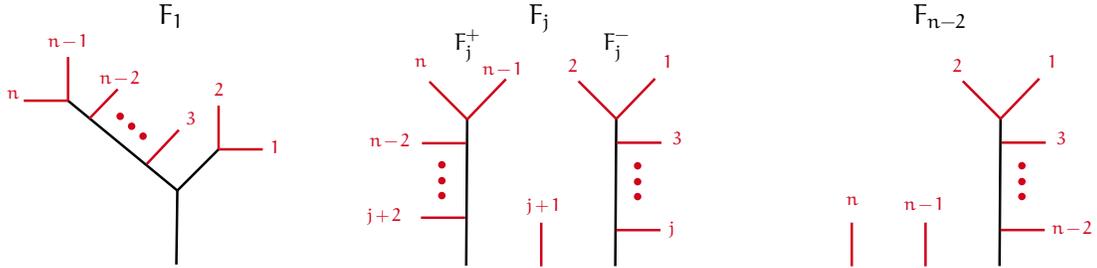}}
    \caption{Marked fragments that attach to the horizontal edge of $\tilde{T}$ to obtain the base graph $T$. We denote by $F_j^-$ the connected component that contains the marks with the lowest indices, and $F_j^+$  the one containing the highest labels.}
    \label{fig:markfrag}
\end{figure}

\subsubsection*{Constructing the covers}

In this section we construct covers $\Gamma \to T$ in $(\Fs\times \Ft)^{-1}(p)$. Given a genus part of the cover ending with an active edge of degree $d-2i$ and a fragment $F_j$ with \begin{equation}\label{eq:range}
  i \leq j \leq n-2-i = d-i,  
\end{equation}  we show how to attach $F_j$ to $\tilde{T}$ to obtain $T$, and how to mark the inverse images of the fragment to obtain the cover $\Gamma$. We discuss several cases.

\noindent{\textsc{Case 1: $i>0$.}}  Given a genus part with active edge of degree $d-2i$ and $j$ in the range specified in \eqref{eq:range}, we complete the genus part to a cover of type $\tilde{\Gamma}_{i,k}\to \tilde{T}_{i,k}$ as described in Section \ref{sec:markings}, for $k = d-i-j$. We first describe the topological type of the cover, and then concern ourselves with the metric information. Refer to  Figure \ref{fig:gensol} to follow the various constructions involved.

\begin{figure}[tb]
    \centering
     \resizebox{1\textwidth}{!}{\input{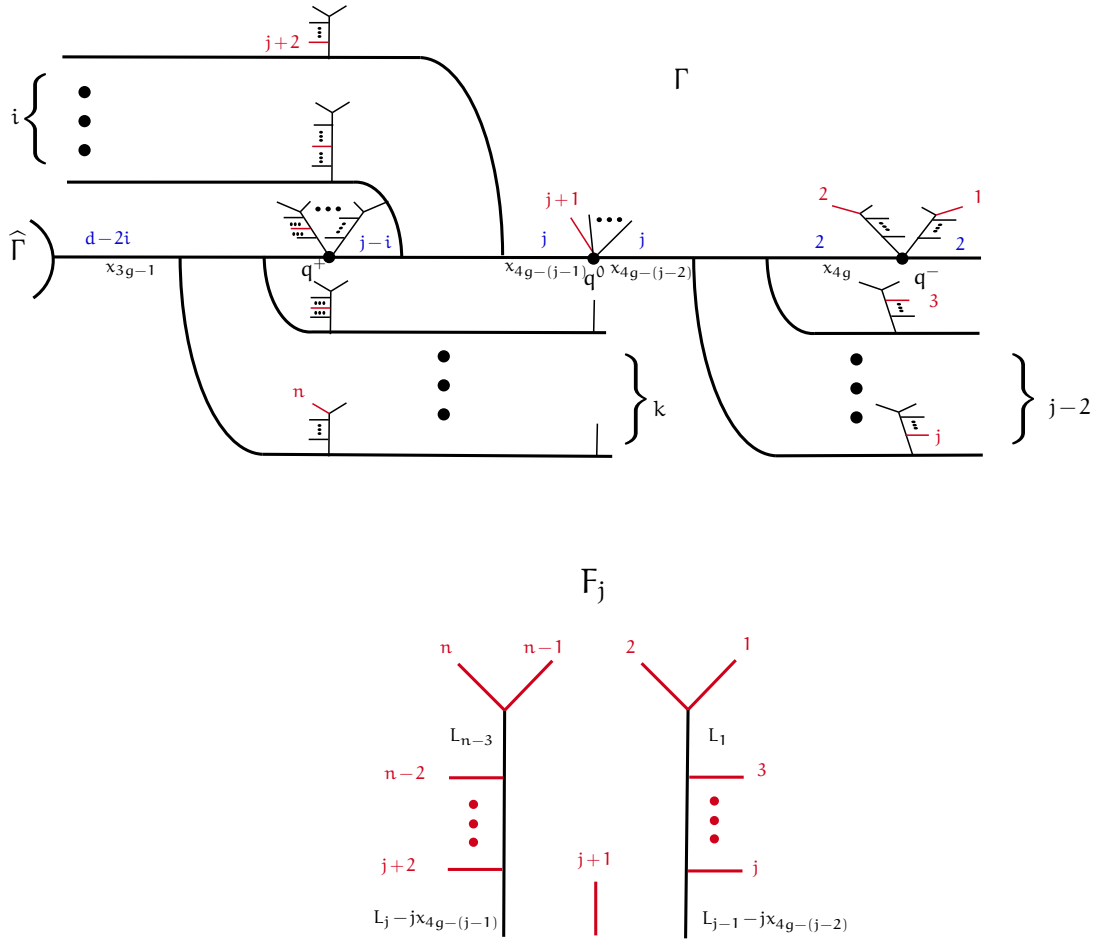}}
    \caption{The top part of the picture represents the tree part of the cover $\phi_{i,j}$ for $i\not=0$. The inverse images of the points $q^\pm,q^0$ on the active path are still denoted by the same names to not clutter the picture. The metric information for the fragment $F_j$ is depicted in the bottom part of the picture as it would not fit above: since all connected components of the inverse image of $F_j$ have local degree one, the lengths are the same in the top graph.}
    \label{fig:gensol}
\end{figure}

 Pick three points on the (image of the) active path\footnote{Since we are concerned only with the tree part of the cover, we now start counting the edges of the active path from right after the last loop is closed.}:
\begin{description}
    \item[$q^+$] on the $(k+1)$-th edge of the active path, separating the first set of cuts from the joins;
    \item[$q^0$] on the $(i+k+1)$-th edge of the active path, separating the joins from the remaining cuts;
    \item[$q^-$] on the last edge of the active path.
    \end{description}

To  obtain ${T_{i,k}}$, we attach $F_j^+$ to the point $q^+$, the (image of the) marked point $j+1$ to $q^0$ and $F_j^-$ to the point $q^-$.

For $\Gamma_{i,k}$, we mark the points in the inverse images of $F_j$ as follows:
\begin{enumerate}
    \item Order the  (connected components of the) inverse images of $F_j$ according to the order of their closest point to the active path;
    \item place the mark $n$ on the first inverse image of $F_j^+$
    \item Place all marks $>2$ in descending order one on each consecutive inverse image of (the appropriate connected component of) $F_j$.
    \item Place the marks $1,2$ on the last inverse image of $F_j^-$, making sure that the two marks are on distinct branches of such inverse image: note that since the point $q^-$ lies on an edge of degree $2$, the inverse image of  $F_j^-$ consists of two copies of $F_j^-$.
\end{enumerate}

We now proceed to describe the metric information. The active path contains $(g+2)$ edges: give the first one length $x_{3g-1}$, and each successive length of $\overline{\Gamma}$ until $x_{4g}$.

For all edges of the fragment $F_j$, except the two bottom edges of $F_j^\pm$, give them the length $L_i$ of the corresponding edge in $\overline{T}$.
The bottom edge of $F_j^-$ is given length $L_{j-1}-jx_{4g-(j-2)}$, and the bottom edge of $F_j^+$ is given length $L_{j}-jx_{4g-(j-1)}$. We have thus constructed a cover in the inverse image of $p$ which we call $\phi_{i,j}$.

\begin{figure}[tb]
    \centering
     \resizebox{1\textwidth}{!}{\input{Figures/i0sol}}
    \caption{The top part of the picture represents the tree part of the cover $\phi_{i,j}$ for $i=0, j\not=1$. The metric information for the fragment $F_j$ is depicted in the bottom part of the picture.}
    \label{fig:gensoli0}
\end{figure}

\vspace{0.2cm}

\textsc{Case 2: $i=0, j>1$.}
This case is illustrated in Figure \ref{fig:gensoli0}. When $i=0$, i.e. the genus part of the cover ends with an active edge of degree $d$, the only option for the cover $\tilde{\Gamma}_{1,j}\to \tilde{T}_{1,j}$ is to be formed by a sequence of $g-1$ cuts from the active edge.

To obtain the (topological type of the) base graph $T_{1,j}$, mark two points $q^+,q^0$, in this order, on the edge after $(n-2-j)$ cuts; place a mark $q^-$ on the last edge of the active path. Attach the fragments $F_j^\pm$ to the points $q^\pm$ and the mark $j+1$ to the point $q^0$.
On the cover graph, the marks are placed on the inverse images of the fragments in descending order as you proceed along the active path. As in the previous case, the marks $1$ and $2$ should be on distinct branches of the inverse image of $F_j^-$.
The metric information is also analogous to the previous case, and it is illustrated in Figure \ref{fig:gensoli0}.

\begin{figure}
    \centering
     \resizebox{\textwidth}{!}{\input{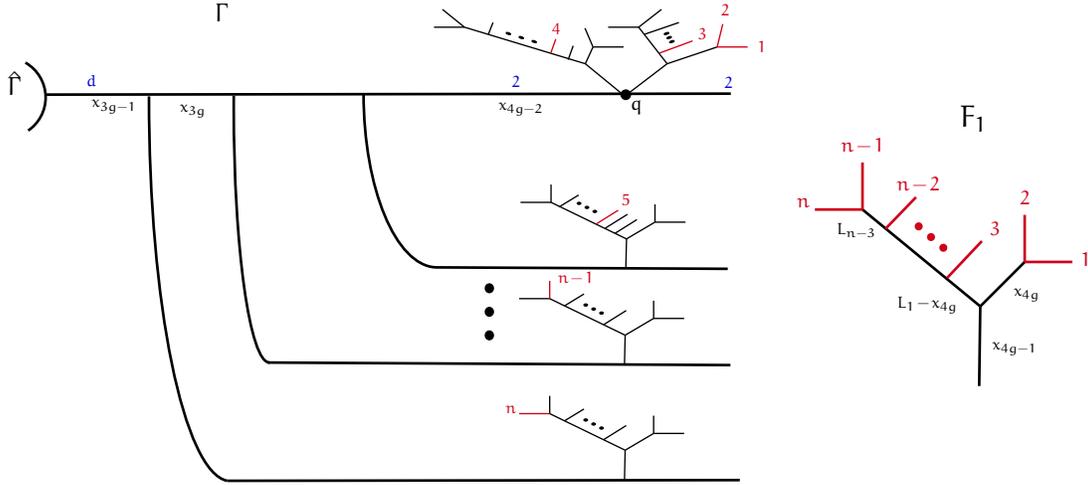}}
    \caption{The left part of the picture represents the tree part of the cover $\phi_{i,j}$ for $i=0, j=1$. The metric information for the fragment $F_1$ is depicted on the right part of the picture.}
    \label{fig:i0j1}
\end{figure}

\vspace{0.2cm}

\textsc{Case 3: $i=0, j=1$.}
This case is illustrated in Figure \ref{fig:i0j1}. Again, the base cover is obtained by performing a sequence of cuts on the active edge. We mark one point on the last edge of the active path, and attach the fragment $F_1$ to it.
On the cover curve, the marks are placed  one on each connected component of the inverse image of $F_1$ in descending order, so that the mark $n$ will stabilize to the vertex of the first cut, and so on. 
The last connected component of the inverse image of $F_1$ attaches to a degree $2$ edge, and therefore consists of two copies of the fragment $F_1$. We mark the point $4$ on one copy, and the points $3,2,1$ on the other.

As for the metric information, consider the path from $\widehat \Gamma$ to the vertex supporting the marks $1,2$, and give the edges of this path lengths $x_{3g-1}, \ldots, x_{4g}$. The remaining edges of the fragment are given lengths $L_1-x_{4g}, L_2, \ldots, L_{n-3}$ as depicted in Figure \ref{fig:i0j1}.

\subsubsection{Multiplicities}
\label{sec:mult}

In this section we compute the local degree of the map $\Fs\times\Ft$ at each of the inverse images of the point $p$, giving the multiplicities we need to count the covers constructed with.
Following \eqref{eq:localdegree}, the local degree is the product of three factors: an automorphism factor, a local Hurwitz numbers factor, and a dilation factor.

There are two types of local Hurwitz numbers appearing in the graphs constructed: either numbers of the form $H_0(\alpha, (2,1^{d-2}), \beta)$, or of the form $H_0((d),(d), 1)$, where this notation means that the third point on the base is not a branch point, but only one of its inverse images is marked (see Convention \ref{conv:lochurnum}). Both these types of Hurwitz numbers are equal to one, and therefore the Hurwitz number factor is also equal to one.

The dilation factor equals the determinant of the matrix $M_g$ giving the local expression for the map $\Fs\times \Ft$; call $x_i, L_j$ the lengths of the edges of the graphs $\overline{\Gamma}, \overline{T}$, and $y_k$ the lengths of $5g$ edges of $\Gamma$ chosen as in observation $(2)$ after condition $(\star)$; then the rows of the matrix express the $y_k$'s as linear functions of $x_i, L_j$. Since all covers can be split into a genus part  and a marked tree part, the matrix used to calculate the dilation factor is block diagonal. The block corresponding to the genus part has size $3g-2$, the other block has size $2g+2$.
\begin{claim}\label{claim:treeblock}
The determinant of the marked tree block equals one.    
\end{claim}
\begin{proof}
     The following two facts are used to calculate the determinant of this block. First, all marked points stabilize to the active path, so for $i\geq 3g-1$ the rows corresponding to $x_i$'s contain exactly one non-zero entry, which is in in fact equal to $1$. We can therefore eliminate the rows and columns containing these entries.
Next, when writing the lengths $L_i$'s in terms of the lengths $y_j$'s, we observe that there is exactly one length of the cover that contributes to $L_i$ and does not lie on the active path. Recall that each cover has one free length for every $L_i$, and the corresponding  entry is $1$. These two facts together give that the absolute value of the determinant is one. 
\end{proof}

\begin{claim}
    The determinant of the block corresponding to the genus part of the cover is equal to $2^g$. 
\end{claim}

\begin{proof}
    
The block that corresponds to the genus part of the cover is itself made of smaller blocks.    We assign  the following matrices, $M_U$ and $M_D$, respectively, to the genus fragments $U$ and $D$. 
\[ M_U=\left[ \begin{array}{cc}
1 & 0 \\
d & 2
\end{array} \right]
M_D=\left[ \begin{array}{cc}
1 & 0 \\
d-1 & 2
\end{array} \right]
\]
The genus block of the matrix
is constructed as a sequence of $1\times 1$ and $2\times 2$ diagonal blocks as follows. 
There is an initial diagonal entry of $2$ corresponding to the first loop.
Then we have a block diagonal 
entry of $M_U$ or $M_D$ depending on which type of genus fragment has been used to form the second genus. Between every loop there is an edge that is part of the active path and thus there is a 1 on the diagonal between any two consecutive $M_U$/$M_D$ blocks.
In conclusion, we have $g-1$ $(2\times 2)$-blocks of determinant $2$, one diagonal entry of $2$ and $g-1$ more diagonal entries of $1$, and the claim follows.
\end{proof}

Finally, we show that the determinant of the matrix $M_g$ equals the automorphism factor. 
The automorphisms of the cover that do not pull back from automorphisms of $\overline{\Gamma}$ correspond to switching pairs of simple branch points attached to the same vertex and their preimages. There is one such pair attached to the first loop, and one for each genus fragment, corresponding to the two ends of the attached tripod.
 All together, $|\Aut(\phi)|/|\Aut(\overline{\Gamma})|= 2^g$.

Putting everything together, for any point $x$  corresponding to a genus $g$ Hurwitz cover $\phi$ that we have constructed, 
\[
\deg_x(\Fs\times\Ft) =\frac{|\Aut(\overline{\Gamma})|}{|\Aut(\phi)|}\cdot|\det(M_g)|\cdot\prod_{v\in V(\Gamma)} H_v = \frac{1}{2^g}\cdot 2^g \cdot 1 = 1.
\]

\subsubsection{Counting solutions} \label{sec:combinatorics}
We now organize the covers constructed in a way that allows us to count them.
We put the covers  $\Gamma \rightarrow T$ inside (but not filling) a rectangular array, where the rows correspond to solutions with the same genus part, and the columns to covers with the same marked fragment type. We order the rows so that the degree of the active edge is non-increasing, and we order the marked fragments by the index $j$ as in Figure \ref{fig:markfrag}. 
It is immediate that this table has $n-2$ columns, while it takes a bit of care to count the number of rows.

\begin{figure}
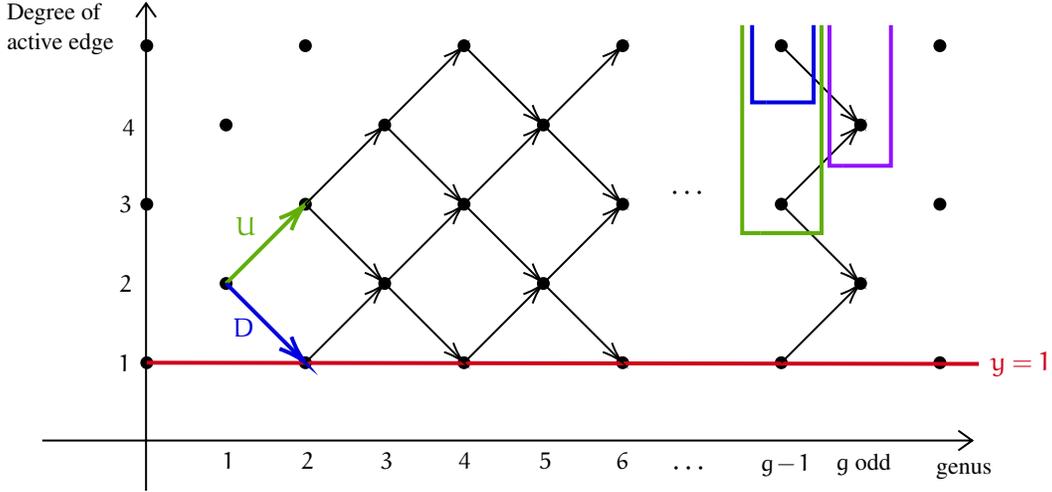

    \centering
    \include{Figures/genusdeggraph}
    \caption{A graph showing for each $g$ which degrees are possible for the active edge by adding $U$ or $D$, green or blue arrow, respectively. The green, blue, and purple boxes demonstrate the proof of Lemma \ref{lemma:deg active}.}
    \label{fig:genusdeggraph}
\end{figure}

Each genus part of a cover corresponds to a  word of length $g-1$ in the letters $U$ and $D$, subject to the condition that at every step the degree of the active end remains positive. There is a natural bijection, illustrated in Figure \ref{fig:genusdeggraph}, between genus parts of the cover and plane paths made by concatenating $g-1$ vectors of type $U = (1,1)$ or $D = (1,-1)$, starting at the point $(1,2)$ and never going below the $y=1$ line. In particular, the $y$-coordinate of the endpoint of a path corresponds precisely to the degree of the resulting active end.

While it is not immediate (at least to us) how to count the number of paths with a given endpoint, it is rather simple to count paths  with endpoint of the form $(g, y)$ with $(g,y_0)$ fixed and $y\geq y_0$; letting $y_0 = d-2i$ this counts the number of genus parts of covers with active end greater or equal than $d-2i$.

\begin{lemma} \label{lemma:deg active}
For $d = g+1 \geq 2$, $0\leq i\leq \lfloor\frac{d-1}{2}\rfloor$, denote by $A_{d, \geq d-2i}$ denote the number of paths described above with endpoint of coordinate $(d, y)$, $y\geq d-2i$. 
\begin{equation}
  A_{d,\geq d-2i}
  = {{d-1}\choose{i}}.
\end{equation}
\end{lemma}
\begin{proof}
    The statement is true by inspection for $d=2$. Next, partition paths in $A_{d,\geq d-2i}$ by the type of the last step. Those with last step $U$ are naturally in bijection with $A_{d-1,\geq d-2i-1}$, and those with last step $D$ are naturally in bijection  with $A_{d-1,\geq d-2i+1} =A_{d-1,\geq d -1-2(i-1)} $; inducting on $d$ we obtain
    \begin{equation}
        A_{d,\geq d-2i} = {{d-2}\choose{i}}+{{d-2}\choose{i-1}} = {{d-1}\choose{i}}.
    \end{equation}
\end{proof}

In Section \ref{sec:markings} we showed that for every genus part of a cover with active edge $d-2i$ we could complete it to a cover with marked fragment type $F_j$ with $i\leq j\leq n-2-i$. Counting the solutions by columns, we see that the columns corresponding to $F_j$ has exactly $A_{d, \geq d-j-1}$ covers. Recalling $d = g+1 = n-2$, we obtain:
\begin{equation}
    |(\Fs\times \Ft)^{-1}(p)| = \sum_{j = 1}^{d} A_{d, \geq d-j-1} = \sum_{m = 0}^{d-1}{{d-1}\choose{m}} = 2^{d-1} = 2^g.
\end{equation}

Since we showed in Section \ref{sec:mult} that each cover counts with multiplicity equal to one, we have thus far shown that $\Tev_g\geq 2^g$. In the next section we complete the proof of Theorem \ref{thm:ttev} by excluding the possibility of any further contributing cover.

\subsection{Excluding further solutions}

\label{sec:exclude}
In this section, we exclude any further cover $\Gamma \to T$ from mapping to  the chosen point $p= (\overline{\Gamma},\overline{T})\in \Mgn{g,n}^\trop\times \Mgn{0,n}^\trop$. We do this by showing no other marked fragments work, there is no other way to form a genus part of the cover with independent cycle lengths, and all joins on $\Gamma$ must occur in a row. 

\subsection*{Fragments attaching to the active edge}
We begin by introducing some notation.
Consider a cover $\phi: \Gamma\to T\in (\Fs\times\Ft)^{-1}(p)$. 

We call  the path connecting the last cycle of $\Gamma$ with the vertex in $\Gamma$ to which the ends marked $1$ and $2$ stabilize in $\overline{\Gamma}$ the {\it active path} of $\Gamma$, and denote it by $AP(\Gamma)$.
 While the stabilization function $T\to \overline{T}$ does not typically admit a global continuous section, such section exists when restricting our attention just to the compact edges of $\overline{T}$. We call this section $\sigma_{\overline{T}}:E(\overline{T})\to T$.

 \begin{lemma}
   The intersection \begin{equation}\label{eq:inter} Im(\sigma_{\overline{T}})\cap \phi(AP(\Gamma)) \end{equation} does not contain the entire image of any edge of $\overline{T}$.
 \end{lemma}

\begin{proof}
 
    One of the aspects of the chosen point  $p$ is that the lengths $L_i$ in $\overline{T}$ are all much longer than the lengths in $\overline{\Gamma}$. The length of the active path in $\Gamma$ equals the sum of the lengths $x_{3g-1}+\ldots +x_{4g}$, and therefore
    the length of its image is bounded  by $d\cdot(x_{3g-1}+\ldots +x_{4g})$, which is by construction less  than any of the $L_i$, thus proving the Lemma. 
\end{proof}    

Since the marked ends stabilize to the active path, it must be that the intersection \eqref{eq:inter} is non-empty, and because the base tree is  trivalent it must be an interval with nonempty interior. Further,  \eqref{eq:inter} must consist either  of an interval of an external edge of $\sigma_{\overline{T}}(\overline{T})$, or an interval containing a marked end; if \eqref{eq:inter} were only a part of a single interior edge $e_i$ of $\sigma_{\overline{T}}(\overline{T})$, by varying the lengths of the two edges in $\Gamma$ adjacent to the intersection one would obtain infinitely many distinct graphs giving length $L_i$ to the stabilization of the edge $e_i$, which is not possible because the point $p$ has been chosen to be in the interior of a maximal cone of (the refinement of) the product $\Mgn{g,n}^\trop\times \Mgn{0,n}^\trop$.
It follows from this discussion that, after reintroducing the marked ends,  the complement 
$Im(\sigma_{\overline{T}})\smallsetminus \phi(AP(\Gamma))$ must be one of the $n-2$ fragments described in Section \ref{sec:markings}.

\subsection*{Splitting of transpositions}

We recall that the  graph $\widetilde{T}$, obtained by forgetting the images of the marked points, has $4g$ ends, all corresponding to simple branched ends for the tropical cover. Borrowing language from the monodromy representation of a Hurwitz cover, we call {\it transpositions} the collection of ends above a branched end which attach to a vertex $v$ with local degrees $2, 1^{d_v-2}$. As we remarked in Section \ref{sec:markings}, removing the simple transpositions and unramified parts of the cover leaves us with a trivalent graph from which the entire cover can be uniquely recovered. In what follows we refer to this trivalent graph, and relevant parts of it, by the names $\Gamma, \widetilde{\Gamma}, \gGamma$. We first analyze how the $4g$ transpositions are split between $\gGamma \to \gT$ and the rest of $\Gamma \to T$. 
\begin{claim}
    The genus part of the graph $\gGamma \to \gT$ contains at least $3g$ transpositions.
\end{claim}

\begin{proof}
    By the Riemann-Hurwitz formula, the smallest number of transpositions to make a graph of genus one is achieved when the degree is equal to $2$; then $4$ transpositions are required. We may use one to attach the subsequent loops,  so we need at least $3$ transpositions to form the first loop.
    After that, we may assume that we have two edges of arbitrary ramification which attach one to the previous, the other to the following loop. Again by the Riemann-Hurwitz formula, if both ends have full ramification, then one can make a cycle with two transpositions. But then the two lengths of such cycle are not independent, see Figure \ref{fig:2transpgenus}.
    Therefore, each loop needs to cover 2 edges of $\gT$, and must use at least 3 transpositions. All together $\gGamma \to \gT$ requires at least $3g$ transpositions. 
\end{proof}

\begin{figure}[tb]
    \centering
    \tikzset{every picture/.style={line width=0.75pt}} 

\begin{tikzpicture}[x=0.75pt,y=0.75pt,yscale=-1,xscale=1]

\draw   (64,74) .. controls (64,62.95) and (79.67,54) .. (99,54) .. controls (118.33,54) and (134,62.95) .. (134,74) .. controls (134,85.05) and (118.33,94) .. (99,94) .. controls (79.67,94) and (64,85.05) .. (64,74) -- cycle ;
\draw [line width=1.5]    (64,74) -- (39.67,74) ;
\draw    (64,74) -- (65,96) ;
\draw    (134,74) -- (133.67,96) ;
\draw [line width=1.5]    (134,74) -- (159.67,74) ;
\draw    (39.67,152) -- (159.67,151.75) ;
\draw    (67.5,151.75) -- (67.5,174.75) ;
\draw    (134.5,151.75) -- (134.5,174.75) ;
\draw    (100.5,116.25) -- (100.5,134.25) ;
\draw [shift={(100.5,136.25)}, rotate = 270] [color={rgb, 255:red, 0; green, 0; blue, 0 }  ][line width=0.75]    (10.93,-3.29) .. controls (6.95,-1.4) and (3.31,-0.3) .. (0,0) .. controls (3.31,0.3) and (6.95,1.4) .. (10.93,3.29)   ;

\draw (45,58.73) node [anchor=north west][inner sep=0.75pt]  [font=\scriptsize]  {$d$};
\draw (62.67,100.4) node [anchor=north west][inner sep=0.75pt]  [font=\scriptsize]  {$2$};
\draw (130.67,100.4) node [anchor=north west][inner sep=0.75pt]  [font=\scriptsize]  {$2$};
\draw (140.67,59.4) node [anchor=north west][inner sep=0.75pt]  [font=\scriptsize]  {$d$};
\draw (91.33,39.73) node [anchor=north west][inner sep=0.75pt]  [font=\scriptsize]  {$i $};
\draw (92,78.4) node [anchor=north west][inner sep=0.75pt]  [font=\scriptsize]  {$d-i $};

\end{tikzpicture}
    \caption{A picture of a genus formed using 2 transpositions. The thickened edges show where the previous and next loops would be attached. When $d=2$  the picture shows the degree $2$ loop, in this case only the thickened edge on the right is connected to other loops. The two lengths $x,y$ of the edges of the loop cover the same bounded edge in $\gT$, therefore they are not independent: they satisfy the relation $i x = (d-i) y$.}
    \label{fig:2transpgenus}
\end{figure}
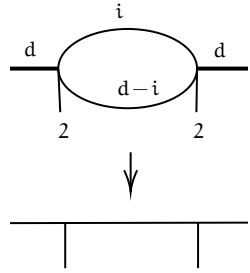

\begin{claim}
    The marked tree part of the graph needs at least $g$ transpositions.
\end{claim}
\begin{proof}
On the marked tree part of $\Gamma \to T$, there are requirements on the lengths associated with the marked fragments. There are $g$ lengths $L_i$'s, each of which is much longer than any of the lengths on the active path; 
thus all marked points  from $3$ to $n-2$ ($g-1$ of them) must be stabilizing to the active path from further down an edge that it is the only mark on. Every time one such edge is formed, at least one transposition must occur. Finally, the branch that the marks $1,2$ lie one must attach to an edge of degree at least $2$, else we would have a relation between the lengths $L_1$ and $x_4g$. All together, at least $g$ transpositions are needed for the marked tree part of $\Gamma \to T$.     
\end{proof}

Since the total number of transpositions is $4g$, the number of transpositions on $\gGamma \to \gT$ is exactly $3g$ and the number on the rest of $\Gamma \to T$ is $g$.

\subsubsection*{Dead ends}
We call {\it dead ends} any part of the graph $\Gamma$ that gets stabilized away in $\overline{\Gamma}$. Because of the fact that we must be using every transposition to obtain some edge length for either $\overline{\Gamma}$ or $\overline{T}$ dead ends can arise in only two ways: they can be individual ends of degree two, or possibly more complicated subgraphs entirely of degree one.

\subsubsection*{Genus part}

\begin{figure}[tb]
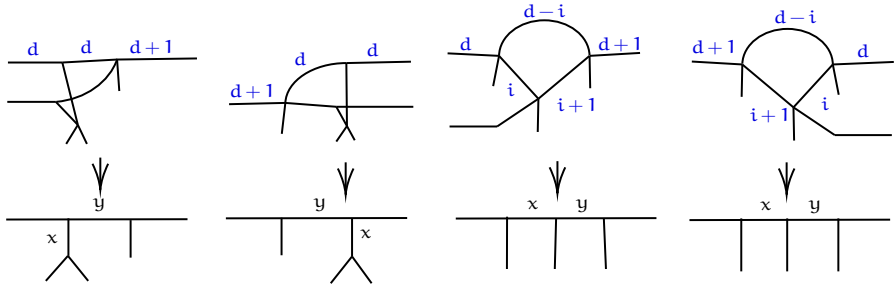

    \centering
    \include{Figures/addinggenus}
    \caption{The four possible ways to form a genus with 3 transpositions. To avoid cluttering the picture we omit many edges and ends of degree one that would be necessary to draw the complete covers. Also note that the horizontal dead ends are unlabeled because they have degree one.}
    \label{fig:addinggenus}
\end{figure}

Knowing that any new loop is added with exactly three transpositions, we look at all possible ways to add a new genus to $\Gamma$ by studying covers of a tree with five ends, two of which may have arbitrary branching data. Due to the graph $\overline{\Gamma}$ needing one much longer edge on each loop, we need to be covering two distinct edges of the base, so no  loops formed by pairs of edges between two vertices work. Distinguishing between an incoming and an outgoing end, which connect to the previous loop and the following one, there are four possible covers forming a loop covering two edges, as shown in Figure \ref{fig:addinggenus}. 
In $\overline{\Gamma}$, we require the 2 lengths of the edges forming each loop to be independent of each other, as one is required to be much longer than the other. Looking at the first picture from the left in Figure \ref{fig:addinggenus}, the loop after stabilization has one edge formed by the horizontal end of degree $d$, the other by the two diagonal ends of degree one together with the curvy edge of degree one. Let us say that the first edge has length $x_t$, the second one $x_b$. With respect to the lengths $x,y$ on the base curve, we have:
\[
x_t = \frac{y}{d}, \ \ \ \ \ 
x_b =  2x+y.
\]
One can see that by choosing $y$ much smaller than $x$ one can make the ratio of the two lengths arbitrarily large.
The situation is identical for the second picture. 

On the other hand, for the third picture let us call $x_t$ the length of the top edge of degree $d-i$ and $x_b$ the length of the edge in the stabilized curved formed by the two edges of degree $i$ and $i+1$. We have
\begin{equation}
\label{eq:leng}    
x_t = \frac{x+y}{d-i}, \ \ \ \ \ x_b = \frac{x}{i}+ \frac{y}{i+1}.
\end{equation}

One can see that \eqref{eq:leng} implies that
\[
\frac{d-i}{d}x_t\leq x_b\leq (d-i) x_t,
\]
showing that the ratio $x_b/x_t$ is bounded both above and below by constants. It is therefore not possible to find an inverse image of the chosen point $p$ that has any loop of this shape. The situation is similar for the fourth picture. In conclusion, every loop must be formed by a fragment of type $D, U$ as described in Section \ref{sec:genuspart}.


\subsubsection*{Tree part}
We now focus our attention on the part of the graph $\Gamma$ that supports the marked points. We have seen so far that after the last loop there is an active path obtained via a sequence of cuts and joins of the active path with edges of degree one. The marks must stabilize, in reverse order,  to each of the trivalent vertices on the active path to obtain $\overline{\Gamma}$. The marked points must lie on inverse images of degree one of some of the fragments $F_j$ described in Section \ref{sec:markings}.

Every time a cut occurs, the degree of the active path decreases by 1, and every time a join occurs, the degree of the active path increases by 1. 
Since joins come from the left and cuts go to the right, in order for both to be covering the same tree of marked points the cut must be to the left of the join. This fact must be true for all cuts and joins covering the first (meaning leftmost) tree in the fragment $F_j$, therefore all cuts must come before all joins that are covering the first tree. After all of the joins, the graph has as many cuts as needed to decrease the degree of the active path to one. Since the fragments $F_j$ have at most two interesting connected components (interesting here means they are not just a single marked end), the second interesting connected component must be stabilizing to only cuts - since any join, occurring to the right of the cuts, would leave an active path of degree greater than one.

We conclude that all covers have a marked tree part that looks like some number of cuts followed by all joins and ending with the rest of the cuts.
But then all possible solutions to our problem must be of the form of those we have exhibited, and there can be no more solutions.
Thus we have concluded the proof of Theorem \ref{thm:ttev}, and established via a direct tropical computation that $\Tev_g = 2^g$.


\bibliographystyle{alpha}
\bibliography{lib}

\end{document}